\documentclass[10pt]{article}

\usepackage{xspace}
\usepackage{ragged2e}
\usepackage{url}
\usepackage{mathtools}
\usepackage{amssymb}
\usepackage{amsmath}
\usepackage{amsthm}
\usepackage{empheq}
\usepackage{latexsym}
\usepackage{enumitem}
\usepackage{eurosym}
\usepackage{dsfont}
\usepackage{appendix}
\usepackage{mathrsfs}
\usepackage{comment}
\usepackage{pifont}
\usepackage{color} 
\usepackage[unicode]{hyperref}
\usepackage{frcursive}
\usepackage[utf8]{inputenc}
\usepackage[T1]{fontenc}
\usepackage{geometry}
\usepackage{multirow}
\usepackage{todonotes}
\usepackage{lmodern}
\usepackage{subcaption}
\usepackage{anyfontsize}
\usepackage{stmaryrd}
\usepackage{bm}
\usepackage{natbib}
\usepackage{cleveref}
\usepackage[english]{babel}
\usepackage[english=british]{csquotes}

\setcitestyle{numbers,open={[},close={]}}

\definecolor{red}{rgb}{0.7,0.15,0.15}
\definecolor{green}{rgb}{0,0.5,0}
\definecolor{blue}{rgb}{0,0,0.7}
\hypersetup{colorlinks, linkcolor={red},citecolor={green}, urlcolor={blue}}
			
\makeatletter \@addtoreset{equation}{section}

\newtheorem{theorem}{Theorem}[section]

\newtheorem{lemma}[theorem]{Lemma}
\newtheorem{proposition}[theorem]{Proposition}

\newtheorem{definition}[theorem]{Definition}
\newtheorem{remark}[theorem]{Remark}

\newcommand{\ol}{\overline}
\newcommand{\ba}{\begin{array}}
\newcommand{\ea}{\end{array}}
\newcommand{\be}{\begin{equation}}
\newcommand{\ee}{\end{equation}}

\def \F{\mathbb{F}}

\def \N{\mathbb{N}}

\def \P{\mathbb{P}}

\def \R{\mathbb{R}}

\def\Cc{\mathcal{C}}

\def\Fc{\mathcal{F}}

\def\Pc{\mathcal{P}}

\def\Tc{\mathcal{T}}

\def\f{\mathbf{D}}

\def\f{\mathfrak{D}}

\def\dbE{\mathbb{E}}
\def\dbF{\mathbb{F}}

\def\dbL{\mathbb{L}}

\def\dbP{\mathbb{P}}
\def\dbR{\mathbb{R}}
\def\dbS{\mathbb{S}}

\def\dbQ{\mathbb{Q}}

\def\a{\alpha}
\def\b{\beta}

\def\d{\delta}
\def\e{\varepsilon}

\def\l{\lambda}
\def\m{\mu}
\def\n{\nu}
\def\si{\sigma}
\def\t{\tau}
\def\f{\varphi}
\def\th{\theta}
\def\o{\omega}

\def\cA{{\cal A}}
\def\cB{{\cal B}}
\def\cC{{\cal C}}

\def\cE{{\cal E}}
\def\cF{{\cal F}}

\def\cL{{\cal L}}

\def\cP{{\cal P}}

\def\cR{{\cal R}}
\def\cS{{\cal S}}
\def\cT{{\cal T}}

\def\cW{{\cal W}}

\def\no{\noindent}
\def\q{\quad}
\def\pa{\partial}
\def\bx{{\mathbf{x}}}
\def\bI{{\mathbf{I}}}
\def\bX{{\mathbf{X}}}
\def\btau{{\bm{\tau}}}
\def\bo{{ \bm{\omega}}}
\def\1{{\mathbf{1}}}
\def\bW{{\mathbf{W}}}

\def \bQ{{\mathbf{Q}}}

\newcommand{\smallfont}[1]{\text{\fontsize{4}{4}\selectfont$#1$}}
\newcommand{\tinyfont}[1]{\text{\fontsize{3}{3}\selectfont$#1$}}

\begin{document}

\title{\bf{Mean-field games of optimal stopping: master equation and weak equilibria}}
\author{Dylan Possamaï\footnote{Department of Mathematics, ETH Zurich, Switzerland, dylan.possamaï@math.ethz.ch} \q Mehdi Talbi\footnote{Department of Mathematics, ETH Zurich, Switzerland, mehdi.talbi@math.ethz.ch}}
\date{\today}

\maketitle

\abstract{We are interested in the study of stochastic games for which each player faces an optimal stopping problem. In our setting, the players may interact through the criterion to optimise as well as through their dynamics. After briefly discussing the $N$-player game, we formulate the corresponding mean-field problem. In particular, we introduce a weak formulation of the game for which we are able to prove existence of Nash equilibria for a large class of criteria. We also prove that equilibria for the mean-field problem provide approximated Nash equilibria for the $N$-player game, and we formally derive the master equation associated with our mean-field game.

\medskip
	\noindent {\bf Keywords:} mean-field games, optimal stopping, master equation.
}
\setlength\parindent{0pt}
\section{Introduction}

{
Introduced independently by \citeauthor*{lasry2007mean} \cite{lasry2006jeux,lasry2006jeux2,lasry2007mean} and \citeauthor*{huang2006large} \cite{huang2003individual,huang2006large,huang2007invariance,huang2007nash,huang2007large}, mean-field games are a powerful tool to study decision problems among large populations of interacting agents. More precisely, these models provide an approximation to study $N$-player games of symmetrical agents, where the players interact only through the empirical measure of their states. In the mean-field approximation, the empirical measure is replaced with an arbitrary probability measure, which makes easier to establish the existence of a Nash equilibrium. The latter typically provides an approximated Nash equilibrium for the $N$-player game.

\medskip
If the classical literature on optimal stopping problems for diffusions as well as their game versions is more than abundant - without any claim for comprehensiveness, let us nonetheless mention \citeauthor*{shiryaev1978optimal} \cite{shiryaev1978optimal}, \citeauthor*{el1981aspects} \cite{el1981aspects}, \citeauthor*{peskir2006optimal} \cite{peskir2006optimal}, \citeauthor*{dynkin1969game} \cite{dynkin1969game}, \citeauthor*{dynkin1969theorems} \cite{dynkin1969theorems}, \citeauthor*{bensoussan1974nonlinear} \cite{bensoussan1974nonlinear,bensoussan1977nonzero}, \citeauthor*{bismut1977probleme} \cite{bismut1977probleme}, \citeauthor*{friedman1973stochastic} \cite{friedman1973stochastic}, \citeauthor*{lepeltier1984jeu} \cite{lepeltier1984jeu}, \citeauthor*{ekstrom2006value} \cite{ekstrom2006value}, \citeauthor*{kobylanski2014dynkin} \cite{kobylanski2014dynkin} -, the corresponding one for mean-field games concentrated for a long time only on optimal controls problems. This is only recently that mean-field games of optimal stopping have received a much stronger attention. We first mention \citeauthor*{carmona2017mean} \cite{carmona2017mean}, where such games are considered with common noise and applied to a bank run model. In particular, the authors introduce a weak formulation of the game, which allows for existence and convergence of equilibria. \citeauthor*{nutz2018mean} \cite{nutz2018mean} introduces a special model of games for a continuum of players, for which existence is characterised by using the cumulative density function of the idiosyncratic noise of the players. A similar comprehensive study is conducted by \citeauthor*{nutz2020convergence} \cite{nutz2020convergence} for a specific game of optimal stopping. On the other hand, \citeauthor*{bertucci2018optimal} \cite{bertucci2018optimal} (see also \citeauthor*{bertucci2021monotone} \cite{bertucci2021monotone}) provides an analytic study of a mean-field game of optimal stopping, by considering the coupled Hamilton--Jacobi--Bellman (here an obstacle problem) and Fokker--Planck equations corresponding to  the problem. It is shown that existence cannot be expected in general for pure stopping strategies, but holds under mild assumptions for relaxed ones. In a similar spirit, the papers by \citeauthor*{bouveret2020mean} \cite{bouveret2020mean,bouveret2022technological} and \citeauthor*{dumitrescu2021control} \cite{dumitrescu2021control,dumitrescu2022energy} consider a probabilistic model for mean-field games of relaxed stopping times, namely through the so-called linear programming approach and the use of occupation measures. We also mention \citeauthor*{huang2022class} \cite{huang2022class} where mean-field optimal stopping games with players comparing their stopping rules are introduced, as well as the recent contribution by \citeauthor*{he2023mean} \cite{he2023mean} which proposes a different approach based on a mean-field extension of a representation result for stochastic processes originally due to Bank and El Karoui.

\medskip
This paper represents our take on mean-field games of optimal stopping. In our modelling, each player faces an optimal stopping problem, and both controlled dynamics and criterion to optimise may depend on the other players' states and stopping decisions. We also allow for the following feature: the players may remain in the game after stopping their process. This feature was already present in the works of \citeauthor*{talbi2023dynamic} \cite{talbi2023dynamic, talbi2023viscosity, talbi2022finite} on the optimal stopping of McKean--Vlasov diffusions, but seems new in the context of games. Indeed, in the aforementioned references, players interact only through the distribution of players still in the game. The first main contribution of our paper concerns the master equation for mean-field games of stopping times, which is introduced to characterise Nash equilibria. Whereas this equation has been extensively studied in the context of games of standard controls (see \emph{e.g.} \citeauthor*{carmona2014master} \cite{carmona2014master}, \citeauthor*{cardaliaguet2019master} \cite{cardaliaguet2019master}, \citeauthor*{mou2020wellposedness} \cite{mou2020wellposedness}, \citeauthor*{gangbo2022mean} \cite{gangbo2022mean}, \citeauthor*{delarue2020master} \cite{delarue2020master} and the references therein), we derive it for the first time in the context of games of timing. Although we can hardly expect its solutions to be differentiable, we only consider in this first approach classical solutions, for which we prove a verification result. The research of an appropriate notion of weak solution is left for further research.

\medskip
Our second main contribution is related to the weak formulation of our mean-field game. We show that, in this setting, the notion of stopping strategy used by the players simply corresponds to the one of randomised stopping time, in the sense of \citeauthor*{baxter1977compactness} \cite{baxter1977compactness} and \citeauthor*{meyer1978convergence} \cite{meyer1978convergence}. Our formulation turns out to be particularly convenient to prove existence of equilibria under very mild requirements---for example, for a large class of so-called `time-inconsistent' problems---and provides approximated equilibria for the corresponding finite population game. {\color{black} Also, it is worth mentioning that our notion of relaxed strategy---through randomised stopping times---is a very classical tool for the study of optimal stopping problems\footnote{\label{foot}The use of randomisation for optimal stopping problems and optimal stopping games was recognised early on, already in discrete-time problems see \citeauthor*{kuhn1953extensive} \cite{kuhn1953extensive}, \citeauthor*{aumann1964mixed} \cite{aumann1964mixed}, \citeauthor*{buckdahn1984randomized} \cite{buckdahn1984randomized}, \citeauthor*{dalang1984arret} \cite{dalang1984arret}, \citeauthor*{yasuda1985randomized} \cite{yasuda1985randomized}, \citeauthor*{nualart1992randomized} \cite{nualart1992randomized}, \citeauthor*{assaf1998optimal} \cite{assaf1998optimal}, and more recently \citeauthor*{chalasani2001randomized} \cite{chalasani2001randomized}, \citeauthor*{rosenberg2001stopping} \cite{rosenberg2001stopping}, \citeauthor*{touzi2002continuous} \cite{touzi2002continuous}, \citeauthor*{solan2012random} \cite{solan2012random}, \citeauthor*{shmaya2014equivalence} \cite{shmaya2014equivalence}, or \citeauthor*{pennanen2018optimal} \cite{pennanen2018optimal}.} and can very intuitively be interpreted as stopping policies, which might not be the case for the other notions of relaxations in the literature related to mean-field games of optimal stopping. For example, \cite{carmona2017mean} relies on an extra level of randomisation, which allows the authors to prove powerful results, such as the convergence of Nash equilibria for the $N$-player game to mean-field equilibria, but somehow loses at the same time some clarity in terms of stopping policy.}

\medskip
The paper is structured as follows. In \Cref{sec:N-player}, we introduce the $N$-player game and the corresponding notion of (approximated) Nash equilibrium. In \Cref{sec:strong-formulation}, we introduce the corresponding mean-field game in strong formulation, for which we extend previous existence results from the literature and derive the master equation. \Cref{sec:weak-formulation} is concerned with the weak formulation of our mean-field games, and with the related existence and convergence results. We also discuss the connection between our formulation and other ones in the literature, and formally derive an analytic formulation of the game. \Cref{sec:examples} applies our theory to some examples and \Cref{sec:proofs} contains the proofs of the main claims of \Cref{sec:weak-formulation}.

\medskip
{\footnotesize\textbf{Notations.} Let $\N$ be the set of non-negative integers and $\N^\star$ the set of positive integers. For any $N\in\N^\star$, and for any $N$-uplet ${\bf v} \coloneqq  (v_1, \dots, v_N)$ of elements of an arbitrary space $E$, and any $k \in [N] \coloneqq  \{1, \dots, N\}$, we denote ${\bf v}^{-k} \coloneqq  (v_1, \dots, v_{k-1}, v_{k+1}, \dots, v_N) \in E^{N-1}$ and, for any $v \in E$, $v \otimes_{k} {\bf v}^{-k} \coloneqq  (v_1, \dots, v_{k-1}, v, v_{k+1}, \dots, v_N) \in E^N$. We also write 
\[
 m^N({\bf v}) \coloneqq  \frac 1 N \sum_{k=1}^N \d_{v_\smallfont{k}},
 \]
the empirical measure associated with the vector ${\bf v}$. Given $p \ge 1$ and a metric space $(E,d)$, $\cP_p(E)$ denotes the space of probability measures on $E$ with finite $p$-th order moment, endowed with the $p$-Wasserstein distance
\[
 \cW_p(\m, \n) \coloneqq   \inf_{\dbQ \in \Gamma(\m,\n)} \bigg\{\int_{E \times E} d^p(x,y) \dbQ(\mathrm{d}x,\mathrm{d}y)\bigg\} ^{1/p},\; \forall (\m, \n) \in E^2, 
 \]
where $\Gamma(\m, \n)$ denotes the set of couplings of $\m$ and $\n$. We also denote by $\cC^d$ the set of continuous paths from $[0,T]$ to $\dbR^d$.}

\section{The game for finitely many players}\label{sec:N-player}

\subsection{Formulation of the problem}

Fix $T >0$, $p \in [1, \infty)$ as well as positive integers $d$, $d^\prime$, and $N$. Let $(\Omega^0, \cF_T^0, \dbF^0 \coloneqq  \{\cF_t^0\}_{0 \le t \le T}, \dbP^0)$ be a filtered probability space endowed with a $d^\prime \times N$-dimensional, $(\dbF^0,\P^0)$--Brownian motion $\bW \coloneqq  (W^1, \dots, W^N)$. We denote $\dbF^\bW$ the filtration generated by $\bW$, and completed under $\P^0$. For $t \in [0,T]$, we denote by $\cT_{t,T}$ the set of $[t,T]$-valued, $\dbF^\bW$--stopping times. Let $\btau \coloneqq  (\t_1, \dots, \t_N) \in \cT_{t,T}^N$, {$\bo \in (\cC^d)^N$} and $\bX^{t, \bo, \btau} \coloneqq  (X^{t, \bo, \btau, 1}, \dots, X^{t, \bo, \btau,N})$ be an $\dbR^{d \times N}$-dimensional process defined by
\begin{align}\label{players-dyn}
\mathrm{d}X_s^{t, \bo, \btau, k} = b_k\big(s, \bX_s^{t, \bo,\btau}\big)I_s^k \mathrm{d}s + \si_k\big(s, \bX_s^{t, \bo, \btau}\big)I_s^k \mathrm{d}W_s^k ,\; k \in [N], \; \mbox{and $\bX_{\cdot \wedge t}^{t, \bo, \btau} = \bo_{\cdot \wedge t}$},
\end{align} 
with $I_s^k \coloneqq  \1_{\{s < \t_\smallfont{k}\}}$, $s\in[t,T]$, $k \in [N]$, and
\begin{align}\label{b-sig-sym}
 b_k(t, \bx) = b\big(t, x_k, m^N(\bx)\big), \; \si_k(t, \bx) =  \si\big(t, x_k, m^N(\bx)\big),\; \mbox{for all $k \in [N]$},
 \end{align}
where $b : [0,T] \times \dbR^d \times \cP_p(\dbR^d) \longrightarrow \dbR$, and $\si : [0,T] \times \dbR^d \times \cP_p(\dbR^d) \longrightarrow\dbR^{d \times  d^\smallfont{\prime}}$ are continuous in $t$ and Lipschitz-continuous in the two other variables (with respect to the $p$-Wasserstein distance $\cW_p$ for the measure-valued one), uniformly in $t$.

\medskip
For any $k\in[N]$, given the stopping strategies of the other players $\btau^{-k}\in\Tc_{0,T}^{N-1}$, the $k$-th player considers the optimisation problem
\begin{align}\label{k-pb-dyn}
V^k(t, \bo ; \btau^{-k}) \coloneqq  \sup_{\t_\smallfont{k} \in \cT_{\smallfont{t}\smallfont{,}\smallfont{T}}} \dbE^{\P^\smallfont{0}}\big[ g\big(\t_k, X_{\t_\smallfont{k}}^{t, \bo, \t_\smallfont{k} \otimes_\smallfont{k} \btau^{\smallfont{-}\smallfont{k}}}, m^N(\btau, \bX_{\t_\smallfont{k}}^{t, \bo, \t_\smallfont{k} \otimes_\smallfont{k} \btau^{\smallfont{-}\smallfont{k}}})\big)\big],
\end{align}
where the map $g : [0,T] \times \dbR^{d} \times \cP_p([0,T] \times \dbR^d) \longrightarrow \dbR$ is Borel-measurable. 

\medskip
We emphasise here that although the $k$-th player leaves the game at time $\t_k$, and therefore receives for payoff a function of the whole vector $\bX$, he only has the power to {\it effectively stop}---or in other words to {\it control}---the $k$-th component $X^k$. By doing so, he will impact the dynamics of the players who will not have exited the game yet. To the best of our knowledge, this is the first time that games where the `frozen' value of the stopped players keeps impacting the dynamics of the other players are considered, in the context of mean-field games and their finite player versions. In the previous literature, the interaction was only between unstopped players, whereas stopped players do not necessarily exit the game in our framework. Note that this special type of interaction can be covered by our model if we allow the interaction terms in the coefficients $b$ and $\si$ to depend on the joint---empirical, in this paragraph---distribution of $\btau$ and $\bX$. This shall be the case when we introduce our mean-field models, but we stick in this section to a simple interaction through the state processes for the sake of clarity.

\medskip
We can now specify what me mean by a Nash equilibrium in this context. We also introduce the notion of $\e$-Nash equilibrium, which will be of interest when we analyse the connection of the $N$-player game with the corresponding mean-field game.
\begin{definition}\label{def:N-Nash}
\rm{(i)} We say $\hat \btau \coloneqq  (\hat \t_1, \dots, \hat \t_N) \in \mathcal{T}_{t,T}^N$ is a \textit{Nash equilibrium} for \eqref{k-pb-dyn} if
\[
V^k(t, \bo ; \hat \btau^{-k}) =  \dbE^{\P^\smallfont{0}}\big[ g\big(\hat \t_k, X_{\hat \t_\smallfont{k}}^{t, \bo, \hat \t_\smallfont{k} \otimes_\smallfont{k} \hat \btau^{\smallfont{-}\smallfont{k}}}, m^N(\hat \btau, \bX_{\hat \t_\smallfont{k}}^{t, \bo, \hat \t_\smallfont{k} \otimes_\smallfont{k} \hat \btau^{\smallfont{-}\smallfont{k}}})\big)\big],\; \forall k \in [N]. 
\]
\rm{(ii)} Let $\e > 0$. We say $\hat \btau \coloneqq  (\hat \t_1, \dots, \hat \t_N) \in \mathcal{T}_{t,T}^N$ is an \emph{$\e$-Nash equilibrium} if
\[
V^{k}(t, \bo,\hat \btau^{-k}) \le  \dbE^{\P^\smallfont{0}}\big[ g\big(\hat \t_k, X_{\hat \t_\smallfont{k}}^{t, \bo, \hat \t_\smallfont{k} \otimes_\smallfont{k} \hat \btau^{\smallfont{-}\smallfont{k}}}, m^N(\hat \btau, \bX_{\hat \t_\smallfont{k}}^{t, \bo, \hat \t_\smallfont{k} \otimes_\smallfont{k} \hat \btau^{\smallfont{-}\smallfont{k}}})\big)\big]+ \e,\; \forall k \in [N].
\]

\end{definition}

\subsection{Pathwise derivatives}\label{sec:derivatives}

As will be shown in the next paragraph, the system of dynamic programming equations related to the above problem involves some path-dependency. We therefore introduce an appropriate notion of pathwise derivatives. Since we are only considering continuous paths, and in particular with diffusion processes, we use the notion defined by \citeauthor*{ekren2014viscosity} \cite{ekren2014viscosity}. Given a normed space $(E, \lVert \cdot \rVert_E) $, we denote by $C^0([0,T] \times (\cC^d)^N, E)$ the set of $\dbF$--progressively measurable and continuous functions from $[0,T] \times (\cC^d)^N$ to $E$, where $(\cC^d)^N$ is equipped with the norm $\lvert \o \rvert \coloneqq  \sup_{t \in [0,T]} \max_{k \in [N]} \lvert \o_t^k \rvert$ for all $\bo \in (\cC^d)^N$.

\medskip
For any $L > 0$, let $\cP_L^N$ denote the set of probability measures on $(\cC^d)^N$ under which the canonical process is a continuous semimartingale with characteristics bounded by $L$. We denote by $u \in C^{1,2}([0,T] \times (\cC^d)^N)$ the set of functions $u : [0,T] \times (\cC^d)^N \longrightarrow \dbR$ such that there exist $\pa_t u \in C^0([0,T] \times (\cC^d)^N, \dbR)$, $\pa_\bo u = (\pa_{\o^1} u, \dots, \pa_{\o^N}) \in (C^0([0,T] \times (\cC^d)^N, \dbR^d))^N$ and $\pa_{\bo \bo}^2 u = (\pa_{\o^\smallfont{1} \o^\smallfont{1}} u, \dots, \pa_{\o^\smallfont{N} \o^\smallfont{N}} u) \in (C^0([0,T] \times (\cC^d)^N, \dbS_d))^N$ such that, for all $\dbP \in \cup_{L > 0} \cP_L^N$, $u$ satisfies
\[
\mathrm{d}u(t,\bo) = \pa_t u(t,\bo)\mathrm{d}t + \pa_\bo u(t,\bo) \cdot \mathrm{d}\bo_t + \frac 1 2 \mathrm{Tr}\big[\pa_{\bo \bo}^2 u(t,\bo) \mathrm{d}\langle \bo \rangle_t\big], \; \mbox{$\dbP$\rm--a.s.} 
\]

\subsection{Dynamic programming equation}

In this section, we assume that for any $k\in[N]$, and any $(t,\mathbf{x})\in[0,T]\times(\Cc^d)^N$, the matrix $\si_k(t,\mathbf{x}) \si_k^\top(t,\mathbf{x})$ is invertible, so that $\dbF^\bW$ is also the filtration generated by $\bX^{\btau}$, see for instance \citeauthor*{soner2011quasi} \cite[Lemma 8.1]{soner2011quasi}. The $\dbF^\bW$--stopping times then write $\t = \t(\bX^\btau)$, and their survival process $I_t = I_t(\bX^\btau) = I_t(\bX^\btau_{\cdot \wedge t})$, $t\in[0,T]$. The dynamics of the players in \Cref{players-dyn} are therefore path-dependent, and the formulation \eqref{k-pb-dyn} of our problem is time-consistent. 

\medskip
We observe that \eqref{k-pb-dyn} is a standard optimal stopping problem, with $\bX^{T \otimes_\smallfont{k} \btau^{\smallfont{-}\smallfont{k}}}$ as state process since the $k$th player is the only one to effectively stop the $k$th component of the state process and is facing, from his point of view, a standard optimal stopping problem. This component may then be considered as unstopped, hence the $T$ in the exponent). Therefore, we may characterise the value function of the $k$-th player by the following dynamic programming equation
\begin{align}\label{N-dpe}
\min\bigg\{ - \bigg(\pa_t + \cL^{k} + \sum_{j \in[N]\setminus\{k\}} I_t^j(\bo) \cL^j\bigg) u^k(t, \bo), u^k(t,\bo) - g(\bo_t) \bigg\} = 0,\; u^k(T, \bo) = g(\bo_T),\; (t,\bo)\in[0,T)\times(\Cc^d)^N,
\end{align}
where 
\[
\cL^{k} \f(t,\bo) \coloneqq  b(t, \bo_t) \cdot \pa_{\o_\smallfont{k}} \f(t,\bo) + \frac 1 2 \mathrm{Tr}\big[(\si \si^\top)(t, \bo_t) \pa_{\o_k \o_k}^2 \f(t,\bo)\big], \;\mbox{for all $(t,\bo)\in[0,T]\times(\Cc^d)^N$ and $k \in [N]$.} 
\]

Note that this is indeed a deterministic function of the initial conditions as we only allow for pure, closed-loop stopping policies.

\subsection{Nash equilibrium}


 Using the solutions of \Cref{N-dpe} for all $k \in [N]$, a Nash equilibrium $\hat \t$ could be characterised via its survival functions $\hat \bI \coloneqq  (\hat I^1, \dots, \hat I^N)$ via
\begin{align}\label{N-Nash}
 \hat I_t^k(\bo) = \begin{cases} 1,\; \mbox{if $\min_{0 \le s \le t} \{ u^k(s,\bo) - g_k(\bo_s)\} > 0$}, \\[1em]
 0,\; \mbox{if $\min_{0 \le s \le t}\big \{- \big(\pa_t + \cL^{k} + \sum_{j \in[N]\setminus\{ k\}} \hat I_s^j(\bo) \cL^j\big) u^k(s, \bo)\big\} > 0,$}
\end{cases}\mbox{for all $(k,t,\bo) \in [N]\times[0,T]\times(\Cc^d)^N$}.
\end{align}
This can be seen as the counterpart of the fixed-point condition arising in $N$-player games of standard controls for the games of stopping times.

\section{Mean-field game in strong formulation}\label{sec:strong-formulation}
In this section, we formulate the mean-field game limit of the $N$-player game from \Cref{sec:N-player}.

{
\subsection{Stopped McKean--Vlasov SDEs and admissible stopping times}

Let $W^0$ be an $(\dbF^0,\P^0)$--Brownian motion. Consider the following SDE
\begin{equation}\label{stoppedMKV}
X_t^\tau = X_0 + \int_0^{t \wedge \t} b\big(s, X^\t_s, \dbP^0_{X^\smallfont{\t}_\smallfont{s}}\big)\mathrm{d}s + \int_0^{t \wedge \t} \si\big(s, X^\t_s, \dbP^0_{X^\smallfont{\t}_\smallfont{s}}\big)\mathrm{d}W^0_s, \ \dbP_{X_0}^0 = \m_0.
\end{equation}
Since the law of the process is involved in the coefficients, and the stopping time $\tau$ is {\it intrinsic} to the dynamics, we call such process a {\it stopped McKean--Vlasov diffusion}. We also call {\it admissible stopping times} the elements of $\cT_{0,T}$ such that \eqref{stoppedMKV} has a unique solution, and we still denote by $\cT_{0,T}$ the set of admissible stopping times. Note that it is straightforward to check that this set is non-empty: it contains, for instance, all the stopping times depending on $W^0$ only. However, the question of the existence and uniqueness of solutions to \eqref{stoppedMKV} remains challenging, and we shall leave it for further research. Such dynamics can be interpreted as a continuous distribution of symmetric interacting agents playing the same stopping policy, and corresponds formally to the mean-field limit of \eqref{players-dyn}. We refer to \citeauthor*{talbi2022finite} \cite{talbi2022finite} for a deeper analysis of propagation of chaos results for stopped McKean--Vlasov SDEs. 
}

\subsection{Formulation and existence of a mean-field equilibrium}

Let $W^0$ be an $(\dbF^0,\P^0)$--Brownian motion. Fix $m \in \cP_p([0,T] \times \cC^d)$, for which we denote by $\m$ the second marginal. Let $X^\m$ the process defined by the, standard, diffusion
\[
\mathrm{d}X_t^\m = b(t, X_t^\m, \m)\mathrm{d}t + \si(t, X_t^\m, \m)\mathrm{d}W_t^0, \; \dbP^0_{X_\smallfont{0}^\smallfont{\m}} = \m_0.
\]
This models the dynamics of a representative player given the macroscopic behaviour $\m$ of the continuous infinity of other players. Given a reward $J : \cT_{0,T} \times \cP_p([0,T] \times \cC^d) \longrightarrow \dbR$, the representative player faces the problem
\begin{align}\label{strong-MFG}
V_0(m) \coloneqq  \sup_{\t \in \cT_{\smallfont{0}\smallfont{,}\smallfont{T}}} J(\t,m).
\end{align}

\begin{definition}\label{def:strong-MFE}{\rm
We say $\hat \t \in \cT_{0,T}$ is a strong mean-field equilibrium $($strong {\rm MFE} for short$)$ if it satisfies
\[
V_0(m) =  J(\hat \t, m),\; \mbox{\rm and} \; \dbP^0_{(\hat \t, X_{\smallfont{\cdot}\smallfont{\wedge} \smallfont{\hat \t}}^\smallfont{\m})} = m. 
\]}
\end{definition}

\no The following existence result is an extension of \citeauthor*{carmona2017mean} \cite[Theorem 3.5]{carmona2017mean}.
\begin{theorem}\label{theorem:existence-strong-MFE}
Assume that

\medskip
{$(i)$} $J$ is upper-semicontinuous in $\t;$

\medskip
{$(ii)$} $J$ is super-modular in $\t$, \emph{i.e.} for all $(\t, \t^\prime) \in \cT_{0,T}^2$ and $m \in \cP_p([0,T] \times \cC^d)$
\[
J(\t \wedge \t^\prime, m) + J(\t \vee \t^\prime, m) \ge J(\t, m) + J(\t^\prime, m);
\]

{$(iii)$} we may endow $\cP_p([0,T] \times \cC^d)$ with a partial order $\preceq$ such that

\smallskip
{$\quad(a)$} $ \cT_{0,T}\ni \t \longmapsto \dbP^0 \circ (\t, X^\t)^{-1} \in \cP_p([0,T] \times \cC^d)$
is non-decreasing with respect to $\preceq$, where $X^\tau$ is the unique strong solution of the stopped McKean--Vlasov diffusion \eqref{stoppedMKV}$;$

\smallskip
{$\quad(b)$} $J$ has non-decreasing differences with respect to $m$ in the following sense: for $(\t,\t^\prime,m,m^\prime)\in \cT_{0,T}^2\times\cP_p([0,T] \times \cC^d)^2$ with $\t\leq \t^\prime$ and $m \preceq m^\prime$, we have
\[
J(m^\prime, \t^\prime) - J(m^\prime, \t) \ge J(m, \t^\prime) - J(m,\t).
\]
Then there exists a strong \emph{MFE}.
\end{theorem}
\begin{proof}
By \cite[Theorem B.2]{carmona2017mean}, $\cT_{0,T}$ is a complete lattice. Since $(\cP_p([0,T] \times \cC^d), \preceq)$ is a partially ordered set, and assumptions $(i)$ and $(ii)$ hold, we may apply Topkis's monotonicity theorem, see for instance \citeauthor*{milgrom1990rationalizability} \cite[page 1262]{milgrom1990rationalizability}, to deduce that 
\[
\Phi(m) \coloneqq  \underset{\t \in \cT_{\smallfont{0}\smallfont{,}\smallfont{T}}}{\mathrm{argmax}}\; J(\t, m), \; m\in\Pc_p([0,T]\times\Cc^d),
\]
is non-decreasing in the strong set order, \emph{i.e.} for all $(m,m^\prime)\in \cP_p([0,T] \times \cC^d)^2$ with $m \preceq m^\prime$, and any $(\t, \t^\prime) \in \Phi(m) \times \Phi(m^\prime)$, we have $\t \wedge \t^\prime \in \Phi(m)$ and $\t \vee \t^\prime \in \Phi(m^\prime)$. In addition, since $J$ is upper-semicontinuous with respect to $\t$, \cite[Theorem 1]{milgrom1990rationalizability} guarantees the existence of a maximal element $\hat \phi(m) \in \Phi(m)$ for all $m\in \cP_p([0,T] \times \cC^d)$. The fact that $\Phi$ is non-decreasing in the strong order implies that $\phi$ is non-decreasing for the partial order $\preceq$. Also, by $(iii)$.$(a)$
\[
 \cT_{0,T} \ni \tau \longmapsto \dbP^0 \circ (\t, X^\t)^{-1} \in \cP_p([0,T] \times \cC^d),
 \]
is non-decreasing. Then $\cT_{0,T} \ni \t  \longmapsto \hat \phi(\dbP^0_{(\t, X^\t)}) \in \cT_{0,T}$ is a monotone map from a complete lattice into itself. Hence, by Tarski's fixed point theorem, see \citeauthor*{tarski1955lattice} \cite[Theorem 1]{tarski1955lattice}, there exists $\hat \t \in \cT_{0,T}$ such that
\[
V_0\big(\dbP^0_{(\hat \t, X^\smallfont{\t})}\big)= J\big(\hat \t, \dbP^0_{(\hat \t, X^\smallfont{\t})}\big). 
\]
Denoting $\hat \m \coloneqq  \dbP^0_{X^{\smallfont{\hat \t}}}$, we see that $X^{\hat \t} = X_{\cdot \wedge \hat \t}^{\hat \m}$, and therefore $\dbP^0 \circ (\hat \t, X^{\hat \t})^{-1} = \dbP^0(\hat \t, X_{\cdot \wedge \hat \t}^{\hat \m})^{-1}$. Thus, $\hat \t$ is a strong MFE according to Definition \ref{def:strong-MFE}.
\end{proof}

Assumption $(iii)$ in \Cref{theorem:existence-strong-MFE} is hard to check in general. In \cite{carmona2017mean}, $m$ is only a distribution on $[0,T]$, so that it is only needed to find a suitable partial order on $\cP([0,T])$, which is naturally done in the sense of the stochastic order, \emph{i.e.} through cumulative density functions. The authors are then able to present in \cite[Proposition 3.6]{carmona2017mean} an example of sufficient condition for $(iii)$.$(b)$ to hold. In our context, it is not clear how to find a partial order on $\cP_p([0,T] \times \cC^d)$ that would fit for the general class of problems covered by \eqref{strong-MFG}. 

\medskip
However, if we know more accurately the nature of the mean-field interaction as well as the objective function, we may be able to derive more precise assumptions to ensure the existence of an MFE. Consider for example the case where (using the Markovian setting for the sake of clarity) 
\begin{align}\label{J-expectation}
J(\t, m) \coloneqq  \dbE^{\P^\smallfont{0}}\big[\phi(B_\t) \psi(\m_\t^2)\big],
\end{align}
where $B$ is a $d$-dimensional Brownian motion, $\phi : \dbR^d \to \dbR$ and $\psi : \dbR^d \to \dbR$ are upper semicontinuous and nonnegative, with $\phi$ convex, $\psi$ nondecreasing, and $\m_t^2 \coloneqq \int_{\dbR^d} x^2 \mathrm{d} \m_t(x)$ for all $t \in [0,T]$. In this case, we may weaken the requirements in \Cref{def:strong-MFE} and say that $\hat \t \in \cT_{0,T}$ is a strong MFE if 
\[
V_0(\m) = J(\hat \t, \m),\; \mbox{and}\; \dbE^{\P^\smallfont{0}}\big[h(X_{t \wedge \hat \t})\big] = \m_t[h],\; \mbox{for all $t \in [0,T],$}
\]
observing that, in this setting, $V_0$ and $J$ depend only on $\m$. We may then prove the existence of a strong MFE. Indeed, introduce the following partial order on $\cP_2(\cC^d)$:
$$ \m \preceq \tilde \m \q \mbox{whenever} \q t \mapsto \psi(\tilde \m_t^2) - \psi(\m_t^2) \ \mbox{is nondecreasing.}$$ 
Assumptions (i) and (ii) of Theorem \ref{theorem:existence-strong-MFE} being clearly satisfied, we focused on Assumptions (iii) (a) and (b). First, for $\t \le \t'$, observe that
$$ \psi\big( \dbE\big[ B_{t \wedge s}^2 \big]|_{s = \t'} \big) - \psi\big( \dbE\big[ B_{t \wedge s}^2 \big]|_{s = \t'} \big) = \psi(t \wedge \t') - \psi(t \wedge \t), $$
which is nondecreasing in $t$ as $\psi$ is nondecreasing. Therefore Assumption (iii) (a) is satisfied. Moreover, for $\m \preceq \tilde \m$, we have as $\phi$ is convex:
$$ (\psi(\tilde \m_t^2) - \psi( \m_t^2))\dbE\big[ \phi(B_t) | \cF_s \big] \ge (\psi(\tilde \m_s^2) - \psi( \m_s^2))\phi(B_s) \q \mbox{for all $s \le t$}, $$
which proves that $t \mapsto  (\psi(\tilde \m_\t^2) - \psi(\m_\t^2))\phi(B_\t)$ is a sub-martingale and implies that Assumption (iii) (b) is satisfied. We may therefore apply Theorem \ref{theorem:existence-strong-MFE} and obtain the existence of an equilibrium.

\subsection{Master equation}

In this section, we denote for all $t \in [0,T]$
\[
\bQ_t \coloneqq  [t,T) \times \dbR^d \times \cP_p(\cC^d), \; \mbox{and} \; \overline\bQ_t \coloneqq  [t,T] \times \dbR^d \times \cP_p(\cC^d). 
\]

We start by defining useful notions of differentiability. 
\begin{definition}\label{C12S} {\rm
\no {$(i)$} $u : \cP_p(\cC^d) \longrightarrow \dbR$ has a functional linear derivative if there exists $\d_m u:  \cP_p(\cC^d) \times \cC^d \longrightarrow \dbR$ satisfying for any $(\m, \n) \in \cP_2(\dbR^d)\times\cP_p(\cC^d) $
\begin{align*}
u(\n)-u(\m) = \displaystyle\int_0^1 \int_{\cC^\smallfont{d}} \d_m u(\l \n + (1-\l)\m, \o)(\n-\m)(\mathrm{d}\o)\mathrm{d}\l,
\end{align*}
and $\d_m u$ has $p$-polynomial growth in $x \in \dbR^d$, locally uniformly in $m \in \cP_p(\cC^d)$.

\medskip
{$(ii)$} For any $t\in[0,T]$, we denote by $C^{1,2,1}_b(\ol \bQ_t)$ the set of bounded functions $u: \ol \bQ_t \longrightarrow \dbR$ such that $\pa_t u$, $\pa_x u$, $\pa_{xx}^2 u$, $\d_m u$, $\pa_\o \d_m u$, $\pa_{\o \o}^2\d_m u$ exist and are continuous and bounded in all their variables, where the pathwise derivatives are defined as in {\rm\Cref{sec:derivatives}}, with $N=1$.}
 \end{definition}

\no Introduce the dynamic version of the problem faced by the representative player:
given $m \in \cP_p([0,T] \times \cC^d)$ with second marginal $\m$ and $\t \in \cT_{0,T}$ with survival process $I$, we denote by $X^{t, \m, \t}$ the stopped McKean--Vlasov diffusion
\[
X_s^{t, \m, \t} = X_t^{t, \m, \t} + \int_t^s b\big(r, X_r^{t, \m, \t}, \dbP^0_{X_\smallfont{r}^{\smallfont{t}\smallfont{,} \smallfont{\m}\smallfont{,} \smallfont{\t}}}\big)I_r(X)\mathrm{d}r + \int_t^s \si\big(r, X_r^{t, \m, \t}, \dbP^0_{X_\smallfont{r}^{\smallfont{t}\smallfont{,} \smallfont{\m}\smallfont{,} \smallfont{\t}}}\big)I_r(X)\mathrm{d}W_r^0, \; s\in[t,T],
\]
such that $\dbP^0_{X_{\smallfont{\cdot} \smallfont{\wedge}\smallfont{t}}^{\smallfont{t}\smallfont{,} \smallfont{\m}\smallfont{,} \smallfont{\t}}} = \m$. This represents the global dynamics of the continuous population of players. The dynamics $X^{t, x, \m,\t}$ of the representative player is then defined by the diffusion
\[
X_s^{t, x, \m, \t} = x + \int_t^s b\big(r, X_r^{t, x, \m, \t}, \dbP^0_{X_\smallfont{r}^{\smallfont{t}\smallfont{,} \smallfont{\m}\smallfont{,} \smallfont{\t}}}\big)\mathrm{d}r + \int_t^s \si\big(r, X_r^{t, x, \m, \t}, \dbP^0_{X_\smallfont{r}^{\smallfont{t}\smallfont{,} \smallfont{\m}\smallfont{,} \smallfont{\t}}}\big)\mathrm{d}W_r^0, \;s\in[t,T].
\] 
Note that, unlike $X^{t, \m, \t}$, $X^{t,x,\m,\t}$ is a classical diffusion. Given the stopping strategy $\t$ of the other players, the single player faces the optimisation problem
\begin{align}\label{MF-pb}
V(t, x, \m ; \t) &\coloneqq  \sup_{\t^\smallfont{\prime} \in \cT_{\smallfont{t}\smallfont{,}\smallfont{T}}} \dbE^{\P^\smallfont{0}}\Big[ g\big(\t^\prime,X_{\t^\smallfont{\prime}}^{t, x, \m, \t}, \dbP^0_{X_s^{t, \m, \t}}|_{ s = \t^\smallfont{\prime}}\big)  \Big],\; (t,x,\mu)\in[0,T]\times\R^d\times\Pc_p(\Cc^d).
\end{align}


Observing that \eqref{MF-pb} is a standard optimal stopping problem---with the particularity that the second component of the state process is a deterministic flow of distributions---, we may derive the corresponding dynamic programming equation for closed-loop strategies, for $(t,x,\mu)\in[0,T)\times\R^d\times\Pc_p(\Cc^d)$:
\begin{align}\label{master-equation}
\min\bigg\{ - \big(\pa_t + \cL^x\big) U(t,x,\m) - \int_{\cC^\smallfont{d}} I_t(\o) \cL^{\o} \d_m U(t,x,\m,\o) \m(\mathrm{d}\o), \; U(t,x,\m) - g(t,x,\m_t) \bigg\} = 0,
\end{align}
with terminal condition $U(T,\cdot) = g(T,\cdot)$, and
\begin{align*}
\cL^x \f(t,x,\m) &\coloneqq  b(t,x,\m) \cdot \pa_x \f(t,x,\m) + \frac 1 2 \mathrm{Tr}\big[(\si \si^\top)(t,x,\m) \pa_{xx}^2 \f(t,x,\m)\big], \\
\cL^\o \psi(t,x,\m,\o) &\coloneqq  b(t,\o_t,\m) \cdot \pa_\o \f(t,x,\m, \o) + \frac 1 2 \mathrm{Tr}\big[(\si \si^\top)(t,\o_t,\m) \pa_{\o \o}^2 \f(t,x,\m, \o)\big].
\end{align*}

\begin{remark}
{\rm (i)} Note that, for simplicity, we assumed that the dependence of $g$ on $m$ is only in $\m$. In the general case where the measure-valued argument in the Agent's criterion is $\dbP^0_{(\t, X_s^{t, \m, \t})}$, the derivative $\d_m U$ would involve two new variables $(s, \o) \in [0,T] \times \cC^d$, and the above equation would write
 \begin{align}\label{master-equationgen}
\min\bigg\{ - \big(\pa_t + \cL^x\big) U(t,x,m) - \int_{[0,T] \times \cC^\smallfont{d}} I_t(\o) \cL^{\o} \d_m U(t,x,m,s,\o) m(\mathrm{d}s,\mathrm{d}\o), \; U(t,x,m) - g(t,x,m_t) \bigg\} = 0, \nonumber
\end{align} 
with $m \in \cP_p([0,T] \times \cC^d)$.  \\
{\rm (ii)} We observe that, given the behavior $\dbP^0_{X^{t, \m, \tau}}$ of the other Players, the representative Agent faces in some sense a Markov stopping problem, as its own state $X^{t,x,\m,\t}$ is a standard (stopped) diffusion. This contrasts with the $N$-Player game, where the problem faced by the $k$th Player is fully path dependent, even in the state process $X^k$. Indeed, the latter depends on the strategies played by the other Players through the dependency in the empirical measure, which in turn may depend on the whole history of $X^k$, generating path dependency for all the coordinates of the state process. 
\end{remark}


\begin{theorem}\label{theorem:verification}
Assume that there exists a classical solution to \eqref{master-equation} with $I = \hat I[U]$, where
 \begin{align*}
 \hat I[U](t, \o, \m) = \begin{cases} \displaystyle 1, \; \mbox{\rm if $\displaystyle\min_{0 \le s \le t} \big\{U(s,\o_s,\m) - g(s,\o_s,\m)\big\} > 0$}, \\[0.8em] \displaystyle 0, \; \mbox{\rm if $\displaystyle\max_{0 \le s \le t}\bigg\{- (\pa_t + \cL^x) U(s,\o_s,\m) - \int_{\cC^\smallfont{d}}\hat I[u](t,\o, \m) \cL^{\o} \d_m U(t,x,\m,\o) \m(\mathrm{d}\o)\bigg \}> 0.$} \end{cases}
\end{align*}
Then, denoting $\hat \t$ the stopping time with survival process $\hat I[u]$,
\[
U(t,x,\m) = V(t,x,\m;\hat \t), \; (t,x,\mu)\in[0,T]\times\R^d\times\Pc_p(\Cc^d),
\]
and $\hat \t$ is a strong {\rm MFE}. 
\end{theorem}

\begin{proof}
Let $\t \in \cT_{t,T}$. We have by Itô's formula (see \citeauthor*{wu2020viscosity} \cite[Theorem 2.7]{wu2020viscosity}\footnote{In the framework of \cite{wu2020viscosity}, $u$ does not depend on $x$, but the extension is straightforward.})
\begin{align*}
U(t,x,\m) &= U\big(\t, X_\t^{t,x,\m,\hat \t}, \hat \m_\t\big) - \int_t^\t (\pa_t + \cL^x) U\big(s, X_s^{t,x,\m,\hat \t}, \hat \m_s\big) \mathrm{d}s - \int_t^\t \si \pa_x U\big(s, X_s^{t,x,\m,\hat \t}, \hat \m_s\big)\mathrm{d}W_s^0  \\
&\quad - \int_t^\t \int_{\cC^\smallfont{d}} \cL^{\o} \d_m U\big(s,X_s^{t,x,\m,\hat \t},\hat \m_s,\o\big) \hat I[U](s,\o,\hat \m) \hat \m_s(\mathrm{d}\o) \mathrm{d}s,
\end{align*}
where $\hat \m_s \coloneqq  \cL(X_s^{t,\m,\hat \t})$ for all $s\in[0,T]$. We deduce from classical localisation arguments, and from the fact that $u$ is a super-solution of \Cref{master-equation} that 
\[
 U(t,x,\m) \ge  \dbE^{\P^\smallfont{0}}\Big[ g\big(\t,X_{\t}^{t, x, \m, \hat \t}, \cL(X_s^{t, \m, \hat \t})|_{ s = \t}\big)  \Big],\; \mbox{for all $\t \in \cT_{t,T}$.} 
 \] 
By definition of $\hat I[u]$, we also obtain by Itô's formula and localisation arguments
\[
U(t,x,\m) = \dbE^{\P^\smallfont{0}}\Big[ g\big(\hat \t_t,X_{\hat \t_\smallfont{t}}^{t, x, \m, \hat \t}, \cL(X_s^{t, \m, \hat \t})|_{s = \hat \t_\smallfont{t}}\big)  \Big],
\]
where $\hat \t_t $ is the $[t,T]$-valued stopping time with survival process $\hat I[u](t,X_t^{t,x,\m, \hat \t}, \hat \m_t) |_{ I_{\smallfont{t}\smallfont{-}} = 1}$. Observing that $\hat \t_0 = \hat \t$, we conclude by taking $t=0$ in the above equality that $\hat \t$ is a strong MFE for \eqref{strong-MFG}.
\end{proof}

\begin{remark}
If we only look for \emph{MFE} among first optimal stopping times $($\emph{i.e.} the first time the state process hits its Snell envelop$)$, then the master equation characterising the \emph{MFE} is, for $(t,x,\mu)\in[0,T)\times\R^d\times\Pc_p(\Cc^d)$
\begin{align}\label{master-equation-first}
\min\bigg\{ - \big(\pa_t + \cL^x\big) U(t,x,\m) - \int_{\cC^\smallfont{d}}  \cL^{\o} \d_m U(t,x,m,\o)\hat I^1[U](t,\o, \m) \m(\mathrm{d}\o),\; U(t,x,\m) - g(x,\m) \bigg\} = 0, \nonumber
\end{align}
with $\hat I^1[u](t,\o, \m) \coloneqq  \mathbf{1}\big(\{ \min_{\smallfont{0} \smallfont{\le}\smallfont{s} \smallfont{\le} \smallfont{t}}\{U(s,\o_\smallfont{s},\m) - g(\o_\smallfont{s},\m_\smallfont{s})\} > 0\}\big)$, and terminal condition $U(T,\cdot)=g$.
\end{remark}
The assumptions of \Cref{theorem:existence-strong-MFE,theorem:verification} have very little chance to be satisfied in general; in particular, one can hardly expect the master equation to have smooth solutions. We therefore need to introduce an appropriate notion of weak solution for \eqref{master-equation}, a task that we leave for future research. For now, we introduce a weak version of our mean-field game, for which existence holds under much milder assumptions. 

\section{Mean-field game in weak formulation}\label{sec:weak-formulation}

\subsection{Formulation of the problem}

Let $\Omega \coloneqq  [0,T] \times \cC^d$ be the canonical space, $(\t, X)$ be the canonical process on $\Omega$, and $I$ be the survival process corresponding to $\t$, \emph{i.e.}, $I_t \coloneqq  \1_{\{t < \t\}}$, for all $t \in [0,T]$. We endow $\Omega$ with the filtration $\dbF \coloneqq  (\cF_t)_{0 \le t \le T}$ generated by $(X, I)$, with marginal filtrations denoted $\dbF^Z \coloneqq  (\cF^Z_t)_{0 \le t \le T}$ for $Z \in \{X, I\}$. Fix also $p \ge 1$ and $\l \in \cP_p(\dbR^d)$. 

\begin{definition}\label{def:P}{\rm
For $m \in \cP_p(\Omega)$, let $\cR(m)$ be the set of $\dbP \in \cP_p(\Omega)$ such that $\dbP_{X_\smallfont{0}} = \l$,
\[
M_\cdot \coloneqq  X_\cdot - \int_0^\cdot b(s, X, m)\mathrm{d}s,\;  \mbox{\rm and} \; N_\cdot \coloneqq  M_\cdot^2 - \int_0^\cdot (\si \si^\top)(s, X, m)\mathrm{d}s,
\]
are $(\dbP, \dbF)$-martingales and $\cF_t^I$ is $\P$--conditionally independent of $\cF_T$ given $\cF_t$, for all $t \in [0,T]$.}
\end{definition}


\begin{remark}\label{rem:independent}{
The fact that $\cF_t^I$ is conditionally independent of $\cF_T$ given $\cF_t$, for all $t \in [0,T]$, is equivalent to the following property: for all $\f : \cC^d \longrightarrow \dbR$ Borel-measurable, bounded and orthogonal to $\cF_t$ $($\emph{i.e.}, $\dbE^\dbP[\f(X)Z] = 0$ for all $\cF_t$-measurable $Z)$, and for any $\psi : [0,T] \longrightarrow \dbR$ continuous and supported on $[0,t]$, we have 
\[
\dbE^{\dbP}\big[\f(X)\psi(\t) \big] = 0,\;  \mbox{\rm for all $t \in [0,T]$}.  
\]
Although $\f$ does not have to be continuous, we can see that the above equality is stable by passing to weak limits. Indeed, assume that $(\dbP^n)_{n\in\N}$ converge weakly to $\dbP$ as $n \longrightarrow \infty$, and that each $\dbP^n$ satisfies the above equality. By Lusin's theorem, see \emph{e.g.} {\rm\citeauthor*{folland1999real} \cite[Theorem 7.10]{folland1999real}}, for all $\e > 0$, there exists a continuous and bounded $\f^\e$ such that $ \dbP[\f \neq \f^\e] < \e. $
Then we have
\begin{align*}
\big| \dbE^{\dbP}[\f(X)\psi(\t)] \big| \le \big| \dbE^{\dbP}[(\f - \f^\e)(X)\psi(\t)] + \dbE^{\dbP}[\f^\e(X)\psi(\t)] \big| 
&\le \e \lVert \psi \rVert_\infty + \Big| \lim_{n \to \infty} \dbE^{\dbP^n}[\f^\e(X)\psi(\t)] \Big| =  \e \lVert \psi \rVert_\infty,
\end{align*}
which provides the desired result by arbitrariness of $\e$.} 
\end{remark}

We now define the mean-field game. Let $J : \cP_p(\Omega) \times \cP_p(\Omega) \longrightarrow \dbR$. Given the aggregated behaviour of the continuous population of agents $m \in \cP_p(\Omega)$, a representative agent faces the problem
\begin{align}\label{eq:MFG}
V(m) \coloneqq  \sup_{\dbP \in \cR(m)} J(\dbP, m).
\end{align}
We also introduce $\cR^\star(m) \coloneqq  \{\dbP \in \cR(m) : J(\dbP, m) = V(m) \}$. 
\begin{definition}[Mean-field equilibrium]\label{def:MFE}{\rm
 We say $m \in \cP_p(\Omega)$ is a mean-field equilibrium $(${\rm MFE} for short$)$ if $m \in \cR^*(m)$.}
\end{definition}

\begin{remark}{
Similarly to the case of pure equilibria, the distribution of $X$ under an \emph{MFE} is that of a stopped McKean--Vlasov diffusion. 
}
\end{remark}

\subsection{Main results}

\begin{theorem}\label{theorem:existence}
Assume that $J$ is continuous in $m$ and concave in $\dbP$. Then there exists an \emph{MFE}. 
\end{theorem}

As we are about to state a convergence result as $N$ goes to infinity, we shall emphasise the dependence of $\bX$ on $N$. Thus, given $\btau \in \cT_{0,T}^N$, we denote by $\bX^{N, \btau} \coloneqq  (X^{1,N,\btau}, \dots, X^{N,N,\btau})$ the process with dynamics as in \Cref{players-dyn}, with symmetric drift and diffusion coefficients \eqref{b-sig-sym}, with initial values $\bX_0^{N, \btau} = (\xi_1, \dots, \xi_N)$, where the $(\xi_k)_{k \in\N^\smallfont{\star}}$ are $\P$--i.i.d. copies of $X_0$.

\begin{theorem}\label{theorem:convergence}
Assume $J$ writes 
\[
J(\dbP, m) = \dbE^{\dbP}\big[g(\t, X, m)\big], 
\]
where $g : \Omega \times \cP_p(\Omega) \longrightarrow \dbR$ is Lipschitz-continuous in $x$ and $m$, and let $\hat m$ be an \emph{MFE} for the mean-field game corresponding to $J$. {Also assume that $\si \si^\top$ is invertible.} Then there exist two sequences $(\btau^N)_{N \in\N^\smallfont{\star}}$ and $(\e_N)_{N\in\N^\smallfont{\star}}$ respectively valued in $\cT_{0,T}$ and $(0,+\infty)$, such that $(\e_N)_{N\in\N^\star}$ decreases to $0$ as $N$ goes to $\infty$, for any $N\in\N^\star$, $\btau^N$ is an $\e_N$--Nash equilibrium for the $N$-player game \eqref{k-pb-dyn}, and $m^N(\btau^N, \bX^{N, \btau}) \underset{N \to \infty}{\longrightarrow} \hat m$, $\dbP$--{\rm a.s.}
\end{theorem}
 
\begin{remark}
The extension of $g$ from a Markovian reward in \eqref{k-pb-dyn} to a path-dependent one above is straightforward.
\end{remark}

Although the question of uniqueness of MFE in terms of controls remains very challenging, we are able to find conditions guaranteeing uniqueness in terms of the value function, in the case where $b$ and $\si$ do not depend on the measure-valued variable, and where the criterion to maximise takes a special form (inspired from \citeauthor*{bouveret2020mean} \cite[Theorem 4.4]{bouveret2020mean}). As expected, this boils down to a(n) (anti-)monotonicity condition.
\begin{theorem}\label{theorem:uniqueness}
Assume $b$ and $\si$ do not depend on $m$, and that $J$ takes the form
\[
J(\dbP, m) \coloneqq  \dbE^{\dbP}\big[g(\t, X)F(m[g])\big],
\] 
for some $g : [0,T] \times \Omega \times \cP_p(\Omega)\longrightarrow \R$ with $p$-polynomial growth in $x$. Assume further that $F$ is non-increasing. Then all the \emph{MFEs} have the same value for the mean-field game \eqref{eq:MFG}.
\end{theorem}
\begin{proof} 
Let $m^1$ and $m^2$ be two MFE. Since $b$ and $\si$ do not depend on the measure-valued argument, we have $\cR(m^1) = \cR(m^2)$. Therefore, the optimality conditions provide
\[
J(m^1, m^1) \ge J(m^1, m^2),\; \mbox{and}\; J(m^2, m^2) \ge J(m^2, m^1).
\]
 Observing that $J(m, m') = F(m'[g])m[g]$ for all $m, m' \in \cP_p(\Omega)$, we then have:
\[
\big(F(m^2[g]) - F(m^1[g])\big)\big(m^2[g] - m^1[g]\big) \ge 0.
\]
Using that $F$ is non-increasing this leads to
\[
\big(F(m^2[g]) - F(m^1[g])\big)\big(m^2[g] - m^1[g]\big) = 0, 
\]
and finally $F(m^2[g]) = F(m^1[g])$ and $m^2[g] = m^1[g]$, which implies that $J(m^1, m^1) = J(m^2, m^2)$.
\end{proof}

\subsection{Connection with randomised stopping times}
{\color{black} As claimed in the introduction, our notion of weak strategy can be very naturally interpreted as a randomised stopping policy. } Recall \citeauthor*{meyer1978convergence}'s definition of randomised stopping times (see \cite{meyer1978convergence} and also \citeauthor*{baxter1977compactness} \cite{baxter1977compactness}): given a filtered probability space $(\Omega^0, \Fc^0,\dbF^0, \dbP^0)$, a randomised $\F^0$--stopping time is defined as a probability measure $\n$ on $[0,T] \times \Omega^0$ (with $T = \infty$ in \cite{meyer1978convergence}) such that

\medskip
{$\quad(i)$} $\n([0,T], \cdot) = \dbP^0;$

\smallskip
{$\quad(ii)$} the non-decreasing process $A : [0,T] \times \Omega^0 \longrightarrow [0,1]$ defined by $A_t(\omega) := \m_\o([0,t))$, where $\m_\cdot : \Omega^0 \to \cP([0,T])$ is the disintegrated measure of $\n$ with respect to $\dbP^0$, is $\dbF^0$-adapted and satisfies $A_T = 1$, $\dbP^0$--a.s.  

\begin{proposition}
Fix $m \in \cP_p(\cC^d)$. Then any $\dbP \in \cR(m)$ is a randomised stopping time on $(\cC^d, \dbF^X, \dbP_X)$, where $\dbF^X$ is the filtration generated by the canonical process $X$.
\end{proposition}
\begin{proof}
Let $\dbP \in \cR(m)$. We obviously have $\dbP([0,T], \cdot) = \dbP_X$. We now prove $(ii)$. Fix $t \in [0,T]$, let $\f : \cC^d \longrightarrow \dbR$ be bounded and orthogonal to $\cF_t$, and $\psi : [0,T] \longrightarrow \dbR$ be supported on $[0,t]$. We have
\begin{align*}
\dbE^\dbP[\f(X)\psi(\t)]  = 0.
\end{align*}
since $\cF_t^I$ is conditionally independent of $\cF_T$ given $\cF_t$ under $\dbP$. As $\f$ and $\psi$ are arbitrary, this implies that the process $A$ is $\dbF$-adapted, and finally $\dbF^X$-adapted as it is a function of $X$ only. Also observe that 
\[
 \dbE^{\dbP_\smallfont{X}}[\f(X)] =  \dbE^\dbP[\f(X)] =  \dbE^{\dbP_\smallfont{X}}[\f(X)A_T],
 \]
which also implies by arbitrariness of $\f$ that $A_T = 1$, $\dbP_X$--a.s. 
\end{proof}

{\color{black}
\begin{remark}
It is well-known that randomised stopping can prove being very helpful in the study of optimal stopping problems, in particular regarding the existence of optimal stopping strategies, especially thanks to their compactness property, see the references in {\rm\Cref{foot}}. The compactness of $\cR(m)$ for all $m \in \cP_p(\Omega)$ plays of course a crucial role in the large scope of our existence result for mean-field equilibria. 
\end{remark}}

\subsection{Comparison with the literature}\label{sec:comparison}

Mean-field games of stopping times have received strong attention over the past few years. In this section, we explain how our results relate to the existant literature. 

\medskip
The work that is the closest to ours is the one by \citeauthor*{carmona2017mean} \cite{carmona2017mean}. In this paper, the authors investigate a mean-field game of stopping times in the presence of common noise, defining two notions of mean-field equilibria, a strong and a weak one. The strong formulation is the same as ours, besides the fact that we do not consider common noise and that we allow for interaction through the distribution of $X$, including in the dynamics of the players. While being close to ours, their weak formulation is significantly different: indeed, using $\tilde \Omega \coloneqq  [0,T] \times \cC^d \times \cP_p([0,T] \times \cC^d)$ as canonical space with canonical process $(\t, W, m)$, they define the problem faced by a representative player by
\[
V_0 \coloneqq  \sup_{\dbP \in \cR} \dbE^\dbP\big[F(\t, W, m)], 
\]
where $\cR$ is roughly speaking the set of probability measures on $\tilde \Omega$ such that $\t$ is a randomised stopping time and $W$ a standard Brownian motion. Here we omitted the common noise for the sake of clarity. Then, $\hat \dbP \in \cR$ is a weak MFE if it is optimal for the above problem, and if the fixed-point condition $m = \hat \dbP[(\t, W) \in \cdot | m]$ holds. The main difference with our setting is that the measure argument $m$ here is randomised (even without common noise), whereas it is always deterministic in our framework. Thus, our notion of weak equilibrium---that the authors of \cite{carmona2017mean} call \emph{strong MFE with weak stopping time}---is stronger than then one of \cite{carmona2017mean}. Note that this weaker notion of MFE naturally provides the convergence of Nash equilibria for the $N$-player games to the MFE, which is not as clear with our stronger definition. However, our notion of stopping time is easier to interpret in terms of stopping policy, as it simply consists in stopping the reward process using a randomised strategy. We believe there is value in ensuring clearer implementability, even this is at the price of not being necessarily able to prove an appropriate convergence of the $N$-player game to its mean-field counterpart. After all, the original intent of \citeauthor*{huang2003individual} \cite{huang2003individual,huang2006large,huang2007nash,huang2007large} and \citeauthor*{lasry2006jeux} \cite{lasry2006jeux,lasry2006jeux2,lasry2007mean} when they introduced mean-field games was essentially to use them to produce $\varepsilon$--Nash equilibria for the $N$-player game, a result which still holds with our formulation, see \Cref{theorem:convergence}.

\medskip
We also mention the work of \citeauthor*{bouveret2020mean} \cite{bouveret2020mean}, who define a notion of relaxed stopping time for mean-field games where the players interact through the states of players who have not exited the game yet. The interaction term appears in the running cost only. Given an initial state $\m_0^\star \in \cP(\dbR^d)$, the representative player faces the following problem:
\[
\sup_{(m_\smallfont{t})_{\tinyfont{0} \tinyfont{\le} \tinyfont{t} \tinyfont{\le} \tinyfont{T}} \in \cA(\m_\smallfont{0}^\smallfont{\star})} \int_0^T \int_{\dbR^\smallfont{d}} f(t, x, \hat \m_t) \m_t(\mathrm{d}x), 
\]
where $\hat \m \coloneqq  (\m_t)_{0 \le t \le T}$ is the response of the other players, and $\cA(\m_0)$ the set of $m \coloneqq  (\m_t)_{0 \le t \le T}$ such that $\m_0 = \m_0^\star$ and
\begin{equation}\label{relaxed-controls-tankov}
 \int_{\dbR^\smallfont{d}} u(0,x)\m_0^\star(\mathrm{d}x) + \int_0^T \int_{\dbR^\smallfont{d}} (\pa_t + \cL^x)u(t,x) \m_t(\mathrm{d}x) \mathrm{d}t \ge 0,\; 
 \end{equation}
for all $u \in C^{1,2}([0,T] \times \dbR^d)$ such that $u \ge 0$ and $(\pa_t + \cL^x)u$ is bounded. This includes in particular the set of \emph{occupation measures}, \emph{i.e.} the set of measures of the form $\m_t(A) \coloneqq  \dbE\big[\1_A(X_t)\1_{\{t \le \t\}}]$ for all $A \in \cB(\dbR^d)$, with $X$ a standard diffusion process and $\t$ a stopping time. 

 We observe that any of our controls $\dbP$ induces an element of $\cA(m_0)$ (recall that $m_0$ is our initial distribution for $X$). Indeed, for all $\dbP \in \cR(\m_0)$ and 
 $u$ as above, we have by Itô's formula:
 \begin{align*}
0 \le \dbE^\dbP\big[u(\t, X_\t)\big] &= \dbE^\dbP\big[u(0, X_0) + \int_0^\t \cL^x u(t, X_t)dt\big] \\
  &= \dbE^\dbP\big[u(0, X_0) + \int_0^T  \cL^x u(t, X_t) I_t dt \big] \\
  &= \int_{\dbR^n} u(0, x)\m_0(dx) + \int_0^T \int_{\dbR^n} \cL^x u(t, x) m_t(dx,1)ds,
 \end{align*}
 which matches the requirement \eqref{relaxed-controls-tankov} with $\m_t := m_t(dx,1)$, where $m$ here denotes the joint law of $X$ and $I$ under $\dbP$.

 
 Therefore, our existence result \Cref{theorem:existence} implies existence in this setting. Note that this work has been extended to mixed control-and-stopping mean-field games in by \citeauthor*{dumitrescu2021control} \cite{dumitrescu2021control}, in which the authors allow for more general criteria and interaction through the coefficients $b$ and $\si$. If we omit the control part, our existence result also implies existence in their formulation. 

\medskip
Note also that the notion of relaxed equilibria closely relates to the notion of mixed strategies introduced by \citeauthor*{bertucci2018optimal} \cite{bertucci2018optimal}, who provides an analytic study of the mean-field games of stopping times when the interaction only occurs through the criterion to optimise. We dedicate the next section to an informal discussion on the representation of our problem by a system of coupled PDEs, and how this relates to \cite{bertucci2018optimal}.

\subsection{Analytic representation of the mean-field game}

Our mean-field game may also be characterised by a system of coupled PDEs. To see this, we can formally derive the Fokker--Planck equation satisfied by the densities of the optimal flow of marginals $\{\hat \m_t\}_{0 \le t \le T}$ along an MFE, whenever the latter exist. In fact, since $\hat \m$ is path-dependent in our context---recall that a stopping time, or rather its survival process, can be seen as a path-dependent control---, we are only able to derive a Fokker--Planck equation for the density of the joint law of a stopped McKean--Vlasov diffusion $\hat m_t \coloneqq  (\hat X_t, \hat I_t) \in \cP_p(\dbR^d \times \{0,1\})$, assuming it is well-defined. 
Indeed, for any smooth $\f : \dbR^d \to \dbR$, we have
\begin{align*}
\dbE\big[\f(\hat X_t) - \f(\hat X_0)\big] &= \dbE\bigg[\int_0^t \cL \f(s,  \hat X_s, \hat m_s) I_s \mathrm{d}s\bigg] =  \int_{\dbR^\smallfont{d}} \int_0^t \cL \f(s, x, \hat m_s) \hat m(s,x,1) \mathrm{d}s \mathrm{d}r,\; t\in[0,T],
\end{align*}
with $\cL \f(s, x, m) \coloneqq  b(s, x, m)\pa_x \f(x) + \frac{1}{2}\mathrm{Tr}[\si^2(s, x, m)\pa_{xx}^2 \f(x)]$, and, for $i\in\{0,1\}$, $\hat m(t,x,i)$ is the density of $\hat m(t,\mathrm{d}x,i)$ with respect to Lebesgue measure. Note that the extension of $b$ and $\si$ into functions of $m \in \cP_p(\dbR^d \times \{0,1\})$ is straightforward. This implies that we have, in the sense of distributions
\begin{align}\label{FP}
\pa_t \big(\hat m(t,x,1) + \hat m(t,x,0)\big) + \mathrm{div}\big(b(t,x,\hat m_t)\hat m(t,x,1)\big) - \frac 1 2 \partial_{xx}\big(\si^2(t,x,\hat m_t) \hat m(t,x,1)\big)= 0, \; \text{\rm for all $(t,x) \in [0,T] \times \dbR^d$.}
\end{align}

Let us now consider a game of the form $J(\dbP,m) = \dbE^\P\big[g(X_{\t}, \m_\t)]$, with $u : [0,T] \times \dbR^d \longrightarrow \dbR$ the corresponding value function and $\m$ the first marginal of $m$, and define for any $t\in[0,T]$, $\cC_t \coloneqq  \{x \in \dbR^d : u(t,x) > g(x, \hat \m_t)\}$, as well as $\cS_t \coloneqq  \cC_t^c$. 

\subsubsection{Pure strategies}  Let $U$ be a solution of the master equation \eqref{master-equation} corresponding to the Nash equilibrium:
$$ \hat I^1[U](t, \o, \m) := \1\big(\{\min_{0 \le s \le t} \big\{U(s,\o_s,\m) - g(s,\o_s,\m)\big\} > 0 \}\big) \ \mbox{for all $(t, \o, \m) \in [0,T] \times \Omega \times \cP_2(\dbR^d)$}.  $$ 
Let $\hat m$ ne the distribution of $(X^{t,x,\m, \hat \t^1}, \hat I^1[U](t, X^{t,x,\m, \hat \t^1}, \hat \m))$, with $\hat \m$ its first marginal, and introduce $u(t,x) := U(t,x,\hat \m_t)$. It is clear by the chain rule that $u$ is a solution to the standard obstacle problem:
$$ \min\{-(\pa_t + \cL^x)u(t,x), u(t,x) - g(x, \hat \m_t)\} = 0, \ u(T, \cdot) = g(\cdot, \hat \m_T). $$

Also, by definition of $\hat I^1[U](t, \o, \m)$ above, $\hat m(t,x,1)$ and $\hat m(t,x,0)$ are respectively supported on $\cC_t$ and $\cS_t$.
Therefore, the Fokker-Planck equation \eqref{FP} provides:
\begin{align}\label{eqm1m0}
\begin{cases}
\displaystyle \pa_t \hat m(t,x,1)  + \mathrm{div}\big(b(t,x,\hat m_t)\hat m(t,x,1)\big) - \frac 1 2 \partial_{xx}\big(\si^2(t,x,\hat m_t) \hat m(t,x,1)\big) = 0,\; \mbox{on $\cC_t$,} \\[0.5em]
\displaystyle \pa_t \hat m(t,x,0) = 0,\; \mbox{on ${\mathrm{int}}(\cS_t)$}.
\end{cases}
\end{align}
The first equation means that the players still in the game diffuse freely in the continuation region; on the other hand, the stopped players are `frozen', and stop diffusing until the end of the game. Also, since $\cC_t$ is non-increasing in $t$ (in the sense of inclusion), newly stopped players may never reach $\mathrm{int}(\cS_t)$, and remain at $\pa \cS_t$ until the end of the game, hence the second equality. Note that this does not mean that the distribution of the stopped players is constant: since $\cS_t$ is closed, we may only write the above equation in its interior, whereas at any time $t$, newly stopped players keep entering $\cS_t$ through its boundary $\pa \cS_t$.

\medskip
Of course, from the point of view of any player, the problem remains a standard optimal stopping one. Therefore, $(u, \hat m)$ satisfies the system
\begin{align}\label{system-pdes}
\begin{cases}
\displaystyle \min\{-(\pa_t + \cL^x)u, u - g\} = 0,\; \mbox{on $[0,T] \times \dbR^d$}, \\[0.5em]
\displaystyle \pa_t \hat m(t,x,1)  + \mathrm{div}\big(b(t,x,\hat m_t)\hat m(t,x,1)\big) - \frac 1 2 \partial_{xx}\big(\si^2(t,x,\hat m_t) \hat m(t,x,1)\big) = 0,\; \mbox{on $\cC_t$,} \\[0.5em]
\displaystyle \pa_t \hat m(t,x,0) = 0,\; \mbox{on ${\mathrm{Int}}(\cS_t)$}, \\[0.5em]
\displaystyle u(T,x) = g(x), \; \hat m(0^-,t,x) = m_{0-}(t,x),\; \mbox{on $[0,T] \times \dbR^d$.}
\end{cases}
\end{align}
Note that in the case where $b=0$ and $\si=1$, this system reduced to $(u, \hat m(\cdot, 1))$ is the same as the one considered by \citeauthor*{bertucci2018optimal} \cite{bertucci2018optimal}, in the case of pure stopping strategies. 

\subsubsection{Randomised strategies} 
Let us provide a more informal discussion on the case where the players are allowed to play with randomised (or mixed) strategies: they no longer necessarily stop when they reach the border $\cS_t$, but decide to stop with some probability. This intuitively means that they can still be diffusing in $\cS_t$; also, the stopped players naturally remain frozen until the end of the game. Therefore, the second equality in \eqref{eqm1m0} becomes $\pa_t m(t,x,0) \ge 0$, as the players may only enter a `stopping position' but never leave it. Since it is still optimal to keep diffusing on $\cC_t$, we obtain
\begin{align}\label{mixed-eqm1m0}
\begin{cases}
\displaystyle \pa_t \hat m(t,x,1) + \mathrm{div}\big(b(t,x,\hat m_t)\hat m(t,x,1)\big) - \frac 1 2 \partial_{xx}\big(\si^2(t,x,\hat m_t) \hat m(t,x,1)\big) = 0,\; \mbox{on $\cC_t$,} \\[0.5em]
\displaystyle \pa_t \hat m(t,x,1) + \mathrm{div}\big(b(t,x,\hat m_t)\hat m(t,x,1)\big) - \frac 1 2 \partial_{xx}\big(\si^2(t,x,\hat m_t) \hat m(t,x,1)\big) \le 0,\; \mbox{on $[0,T] \times \dbR^d$,}
\end{cases}
\end{align}
where we used \eqref{FP} and the fact that $\pa_t m(t,x,0) \ge 0$ to derive the inequality. Formally, the PDE system characterising equilibria becomes
\begin{align}\label{mixed-system-pdes}
\begin{cases}
 \min\{-(\pa_t + \cL)u, u - g\} = 0,\; \mbox{on $[0,T] \times \dbR^d$}, \\
\displaystyle \pa_t \hat m(t,x,1)  + \mathrm{div}\big(b(t,x,\hat m_t)\hat m(t,x,1)\big) - \frac 1 2 \partial_{xx}\big(\si^2(t,x,\hat m_t) \hat m(t,x,1)\big) = 0,\; \mbox{on $\cC_t$,} \\[0.5em]
\displaystyle \pa_t \hat m(t,x,1) + \mathrm{div}\big(b(t,x,\hat m_t)\hat m(t,x,1)\big) - \frac 1 2 \partial_{xx}\big(\si^2(t,x,\hat m_t) \hat m(t,x,1)\big) \le 0,\; \mbox{on $[0,T] \times \dbR^d$,} \\
 u(T,x) = g(x), \; \hat m(0^-,t,x) = m_{0-}(t,x),\; \mbox{on $[0,T] \times \dbR^d$,}
\end{cases}
\end{align}
which has the same form as the system of PDEs considered in \cite{bertucci2018optimal} for mixed strategies. Note that the term $\pa_t m(t,x,0)$ can be interpreted as the potential $V$ discussed by the author, see the comments after \cite[ Definition 1]{bertucci2018optimal}), in the time-dependent case. Note that the authors of \cite{bouveret2020mean} rigorously prove the equivalence of this system with their notion of MFE; therefore, from our discussion in paragraph \ref{sec:comparison}, existence of a weak equilibrium in our sense implies existence of a weak equilibrium in the sense of \cite{bertucci2018optimal}.

\subsection{Connection with the mean-field optimal stopping problem}

We briefly discuss the connection between our mean-field game of stopping times and the so-called mean-field optimal stopping problem, introduced by \citeauthor*{talbi2023dynamic} \cite{talbi2023dynamic}. This connection is already well-known in the case of {\it potential} mean-field games of standard controls, which relate to an auxiliary McKean--Vlasov control problem, see \emph{e.g} \citeauthor*{briani2018stable} \cite{briani2018stable} or \citeauthor*{cecchin2022weak} \cite{cecchin2022weak}. 

\medskip
Consider the simple case where $b$ and $\si$ do not depend on the measure-valued argument, and where 
\begin{align}\label{potentialMFG}
 J(\dbP, m) \coloneqq  \dbE^\dbP\big[g(\t, X, m)\big],\; \mbox{for all $(\dbP, m) \in \cP_p^2(\Omega)$}, 
 \end{align}
where $g : \Omega \times \cP_p(\Omega) \longrightarrow \dbR$ is such that there exists $G : \cP_p(\Omega) \longrightarrow \dbR$ with $g = \d_m G$. In this case, the mean-field game is said to be {\it potential}. We next introduce the mean-field optimal stopping problem associated with $G$, in a weak formulation setting similar to \cite{talbi2023dynamic}. Since $b$ and $\si$ do not depend on $m$, we have $\cR(m) = \cR(m^\prime)$ for all $(m, m^\prime) \in \cP^2_p(\Omega)$, and we may simply denote by $\cR$ their common value. Then, the mean-field optimal stopping problem may simply write
\begin{align}\label{MFoptstop}
\sup_{\dbP \in \cR} G(\dbP). 
\end{align}
We then have the following result.
\begin{proposition}\label{prop:MFG-MFC}
The mean-field optimal stopping problem \eqref{MFoptstop} admits a maximiser. Moreover, any maximiser provides an \emph{MFE} for the potential mean-field game \eqref{potentialMFG}.
\end{proposition}
\begin{proof}
Since $G$ is differentiable, it is continuous. Moreover, by \Cref{prop:compact}, $\cR$ is compact, and therefore \eqref{MFoptstop} has a non-empty set of maximisers. Let $\hat \dbP \in \cR$ denote one of them. Fix also $\l \in (0,1)$ and $\dbP \in \cR$. We have, by optimality of $\hat \dbP$
\[
 \l \big(G(\hat \dbP) - G(\dbP)\big) \ge 0, 
 \]
which, by regularity of $G$ and definition of $\d_m G$, implies
\[
 \l \int_\Omega g(\t, X, \hat \dbP)d(\hat \dbP - \dbP)(\t, X) + \circ(\l) \ge 0,  
 \]
which, after dividing by $\l$ and letting $\l \longrightarrow 0$, leads to $J(\hat \dbP, \hat \dbP) = \dbE^{\hat \dbP}\big[g(\t, X, \hat \dbP)\big] \ge \dbE^{\dbP}\big[g(\t, X, \hat \dbP)\big] = J(\dbP, \hat \dbP).$ By the arbitrariness of $\dbP \in \cR = \cR(\hat \dbP)$, this proves that $\hat \dbP$ is an MFE for \eqref{potentialMFG}.
\end{proof}

\section{Examples}\label{sec:examples}

In this section, we mention some forms of our criterion $J$ for which the assumptions of Theorem \ref{theorem:existence} are satisfied. 

\subsection{Concave function of the expectation}
Let $\f : \dbR \longrightarrow \dbR$ be concave and $f : [0,T] \times \dbR^d \times \cP_p(\Omega) \longrightarrow \dbR$ be continuous, and consider the case where
\[
J(\dbP, m) \coloneqq  \f\big(\dbE^{\dbP}\big[f(\t, X_\t, m)\big]\big),\; \mbox{for all $\dbP, m \in \cP_p(\Omega)$}. 
\] 
It is straightforward to see that $J$ is continuous. Moreover, since $\f$ is concave and $\dbP \longmapsto \dbE^\dbP[\cdot]$ is linear, $J$ is concave in $\dbP$. Therefore Theorem \ref{theorem:existence} applies.

\subsection{Concave probability distortion}
Let $\f$ and $g$ be defined as above, with the additional assumption that $g$ is non-negative. Consider the case where 
\[
J(\dbP, m) \coloneqq  \cE^\dbP\big(f\big(X_\t, m)\big),\;\mbox{for all $\dbP, m \in \cP_p(\Omega)$}, 
\]
where $\cE^\dbP$ is a nonlinear expectation defined by
\[
 \cE^\dbP(\xi ) \coloneqq  \int_0^\infty \f\big(\dbP[\xi \ge x]\big)\mathrm{d}x,
 \]
for all non-negative random variable $\xi$. This reward has been studied in the context of optimal stopping by \citeauthor*{xu2013optimal} \cite{xu2013optimal}, for various shapes of the function $\f$ (convex, concave, S-shaped and reversed S-shaped). If $\f$ is concave, we easily see that, for all $(\dbP, \dbP^\prime) \in \cP^2_p(\Omega)$ and $\l \in [0,1]$
\begin{align*}
 \cE^{\l \dbP + (1-\l)\dbP^\smallfont{\prime}}\big(\xi \big) = \int_0^\infty \f\big(\l \dbP[\xi \ge x] + (1-\l)\dbP^\smallfont{\prime}[\xi \ge x] \big)\mathrm{d}x 
 &\ge \int_0^\infty \big(\l\f\big(\dbP[\xi \ge x]\big) + (1-\l)\f\big(\dbP^\smallfont{\prime}[\xi \ge x] \big)\big)\mathrm{d}x \\
 &= \l \cE^\dbP(\xi) + (1-\l) \cE^{\dbP^\smallfont{\prime}}(\xi),
 \end{align*}
 which implies that $J$ is concave in $\dbP$. Since it is also clearly continuous, \Cref{theorem:existence} applies again.
 
\subsection{Non-linear expectations}
Let $g : [0,T] \times \dbR \times \dbR^d \longrightarrow \dbR$. We enlarge the canonical space to $\Omega \coloneqq  [0,T] \times \cC^d \times \cC^{d'}$, with canonical process $(\t, X, W)$, and that we require $W$ to be a standard Brownian motion under each element of $\cR(m)$, for all $m \in \cP([0,T] \times \cC^d)$. Consider now the case where $J$ is given by
 \[
 J(\dbP, m) \coloneqq  \cE_g^\dbP\big[f(\t, X_\t, m)\big], 
 \]
 where $\cE_g^\dbP[f(\t, X_\t, m)] \coloneqq  Y_0^{\dbP}$, with $(Y^\dbP, Z^\dbP)$ the solution of the backward SDE
 \[
  Y_t^\dbP = f(\t, X_\t, m) + \int_t^T g(s, Y_s^\dbP, Z_s^\dbP)\mathrm{d}s - \int_t^T Z_s^\dbP\cdot \mathrm{d}W_s, \\; \mbox{$\dbP$--a.s.} 
  \]
 
 \begin{proposition}\label{prop:G-expectation}
 Assume that $f$ is continuous, $g$ is Lipschitz-continuous in $(y,z)$ and continuous in $t$, $g$ is concave in $(y,z)$, and its concave conjugate $g^\star$ is also continuous. Then there exists an {\rm MFE} for the above problem.
 \end{proposition}
 
 \begin{proof}
 \textbf{Step $1$.} We first prove that $\cE_g^\dbP[\cdot]$ is linear in $\dbP$ when $g$ is linear, \emph{i.e.}, when there exist predictable processes $\a,$ $\b$, and $g^0$ such that $ g(t, y, z) = g^0(t) + \a_t  y + \b_t \cdot z,\; \mbox{for all $(t,y,z) \in [0,T] \times \dbR \times \dbR^d$.}$ In this case, we may use the following representation formula for all $\xi \in \dbL^2(\Omega, \cF_T)$
 \[
  \cE_g^\dbP[\xi ] = \dbE^\dbP\bigg[\Gamma_T \xi + \int_0^T \Gamma_t g^0(t)\mathrm{d}t \bigg], 
  \]
 with 
 \[
 \Gamma_t \coloneqq  \exp\bigg(\int_0^t \b_s \cdot \mathrm{d}W_s - \int_0^t \bigg(\a_s - \frac{\| \b_s \|^2}{2}\bigg)\mathrm{d}s\bigg),
 \]
 which clearly proves that $\cE_g^\dbP[\cdot]$ is linear in $\dbP$. 
 
 \medskip
 \no\textbf{Step $2$.} Let us now move to the general case. Since $g$ is concave in $(y,z)$, we may introduce its concave conjugate $g^\star$, which satisfies
 \[
 g(t,y,z) = \inf_{(\a, \b) \in D_{\smallfont{g^\tinyfont{\star}}\smallfont{(}\smallfont{t}\smallfont{,} \smallfont{\cdot}\smallfont{)}}} \big\{ g^\star(t,\a,\b) + \a y + \b \cdot z \big\},
  \]
 where $D_{g^\smallfont{\star}(t, \cdot)}$ denotes the domain of $g^\star(t, \cdot)$ for all $t \in [0,T]$. Also denote by $\cE_{g^\smallfont{\star}, \a, \b}^\dbP[\cdot]$ the non-linear expectation associated with the linear generator $(t, y, z) \longmapsto g^\star(t,\a,\b) + \a y + \b \cdot z$. Then, by \citeauthor*{el1997backward} \cite[Proposition 3.4]{el1997backward}, we have 
 \[
 \cE_{g}^\dbP[\xi] = \inf_{(\a, \b) \in \cA^\smallfont{\dbP}} \cE_{g^\smallfont{\star}, \a, \b}^\dbP[\xi], 
 \]
 where $\cA^\dbP$ is the set of $K$-valued, $\F$-predictable processes $(\a, \b)$ such that $\dbE^\dbP\big[\int_0^T g(t, \a_t, \b_t)^2 \mathrm{d}t\big] < \infty,$, where $K$ is a bounded set. Since we assumed that $g^\star$ is continuous, we easily see that $\cA$ is independent of the chosen $\dbP$. Therefore, using the result of \textbf{step $1$}, we conclude that $\cE_{g}^\dbP[\xi]$, and therefore $J$, is concave in $\dbP$. The continuity in $m$ is easily deduced from standard BSDEs regularity results. 
 \end{proof}

\section{Proofs of the main results}\label{sec:proofs}

\subsection{Proof of Theorem \ref{theorem:existence}}

\textbf{Step $1$.} Assume $b$ and $\si$ are bounded. By \Cref{lem:hemicontinuity} and the fact that $J$ is continuous, we deduce from Berge's theorem, see \citeauthor*{aliprantis2006infinite} \cite[Theorem 17.31]{aliprantis2006infinite}, that $\cP_p(\Omega) \ni m \longmapsto \cR^\star(m)$ is upper-hemicontinuous. It has also nonempty (by concavity hence continuity of $J$ in $\dbP$), compact (closed subspace of a compact space) and convex (by concavity of $J$ in $\dbP$ again) values. 

\medskip
Let $L \coloneqq  \lVert b \rVert_{\infty} \vee \lVert \si \rVert_{\infty}$, and $\hat \cP_L \subset \cP_p(\Omega)$ be the set of probability measures $\dbP$ on $[0,T] \times \cC^d$ such that the second marginals are continuous semimartingale measures with characteristics uniformly bounded by $L$, and satisfy $\dbP_{X_\smallfont{0}} = \l$. $\hat \cP_L$ is a nonempty, $\cW_p$-compact and convex subset of the set of bounded signed measures on $[0,T] \times \cC^d$, which is a locally convex Hausdorff space when endowed with the topology of weak convergence. Note also that since $\hat \cP_L$ is $\cW_p$-compact, it is also weakly compact. The set-valued map $\cR^\star : \hat \cP_L \longrightarrow 2^{\hat \cP_\smallfont{L}}$ then satisfies all the assumptions of Kakutani--Fan--Glicksberg's theorem (see \citeauthor*{fan1952fixed} \cite[Theorem 1]{fan1952fixed}), which provides the existence of an MFE.

\medskip
\no\textbf{Step $2$.} We no longer assume that $b$ and $\si$ are bounded. For $n \in\N^\star$, let $(b_n, \si_n) \coloneqq  (- n \vee b \wedge n, -n \vee \si \wedge n)$, {where this expression has to be understood coordinate-wise}. For $m \in \cP_p(\Omega)$, we denote by $\cR_n(m)$ the counterpart of $\cR(m)$ from \Cref{def:P} with $b_n$ and $\si_n$ instead of $b$ and $\si$. Note that $b_n$ and $\si_n$ still satisfy the Lipschitz-continuity assumptions of $b$ and $\si$, and that they are bounded. Therefore, there exists a sequence of MFE $(\dbP^n )_{n \in\N^\smallfont{\star}}$ for the mean-field games
\[
 V^n(m) \coloneqq  \sup_{\dbP \in \cR_\smallfont{n}(m)} J(\dbP, m), \; n \in \N^\star. 
 \]
By compactness of $[0,T]$, $( \dbP^n_{\t} )_{n \in\N^\smallfont{\star}}$ is tight. Similarly, by Lipschitz-continuity of $b$ and $\si$, we easily see by using Aldous's criterion (see for instance \citeauthor*{billingsley1999convergence} \cite[Theorem 16.10]{billingsley1999convergence}) that $(\dbP^n_{X})_{n \in\N^\smallfont{\star}}$ is tight, and thus $(\dbP^n)_{n \in\N^\smallfont{\star}}$ is tight. After possibly passing to a subsequence, we may then assume that $(\dbP^n)_{n\in\N^\smallfont{\star}}$ converges weakly to some $\dbP \in \cP_p(\Omega)$ as $n \to \infty$. Note that, since the sequence $(\dbP^n)_{n \in\N^\smallfont{\star}}$ is $p$-uniformly integrable, \emph{i.e.} 
\[
\lim_{R \to \infty} \sup_{n \in\N^\smallfont{\star}} \dbE^{\dbP^\smallfont{n}}\big[ \lVert X \rVert_\infty^p \1_{\lVert X \rVert_\infty \ge R} \big] = 0,
\]
this also implies the convergence in $\cP_p(\Omega)$. We easily see that the martingale problem from \Cref{def:P} passes to the limit. We conclude by \Cref{rem:independent} that $\dbP \in \cR(\dbP)$. Note that the fixed-point condition for $\dbP$ is immediately provided.

\medskip
\no It remains to prove that $\dbP \in \cR^\star(\dbP)$. {First observe that, by extending the canonical space to $[0,T] \times \cC^d \times \cC^{d^\smallfont{\prime}}$ and the canonical process to $(\t, X, W)$, we may require $X$ to be a diffusion process driven by $W$, with $W$ being a Brownian motion under each element of $\cR(m)$}. Then, we may take $\dbP' \in \cR(\dbP)$ and $(\tilde \Omega, \tilde \cF_T, \tilde \dbF \coloneqq  (\tilde \cF_t)_{0 \le t \le T}, \tilde \dbP)$ a filtered probability space endowed with $\tilde \t$, $\tilde X$ such that $\tilde \dbP_{(\tilde \t, \tilde X)} = \dbP'$ and 
\[
 \tilde X_t = \tilde X_0 + \int_0^t b(s, \tilde X, \dbP) \mathrm{d}s + \int_0^t \si(s, \tilde X, \dbP) \mathrm{d}\tilde W_s, \; \mbox{$\tilde \dbP$--a.s.,} 
 \]
where $\tilde W$ is an $(\tilde \dbF,\tilde \dbP)$--Brownian motion. For $n \in\N^\star$, we introduce $\tilde X^n$, defined as the unique strong solution to the SDE
\[
 \tilde X_t^n = \tilde X_0 + \int_0^t b_n(s, \tilde X^n, \dbP^n) \mathrm{d}s + \int_0^t \si_n(s, \tilde X^n, \dbP^n) \mathrm{d}\tilde W_s. 
 \]

We then define $\tilde \dbP^n \coloneqq  \tilde \dbP_{(\tilde \t, \tilde X^n)}$, which is clearly in $\cR_n(\dbP^n)$. Note that
\[
\cW_p^p(\tilde \dbP^n, \dbP') = \cW_p^p\big( \tilde \dbP_{(\tilde \t, \tilde X^\smallfont{n})}, \tilde \dbP_{(\tilde \t, \tilde X)}\big) \le C_p\dbE^{\tilde \dbP}\big[ \lVert \tilde X^n - \tilde X \rVert_\infty^p \big],
\] 
for some constant $C_p \ge 0$, and that the right-hand side term tends to $0$ as $n \to \infty$, by standard SDEs estimates, due to the fact that $\cW_p(\dbP^n, \dbP) \longrightarrow 0$ as $n \to \infty$. Then $\cW_p(\tilde \dbP^n, \dbP') \longrightarrow 0$ as $n \to \infty$. Thus, by continuity of $J$
\begin{align*}
J(\dbP, \dbP) = \underset{n \to \infty}{\lim} J(\dbP^n, \dbP^n) 
\ge \underset{n \to \infty}{\lim} J(\tilde \dbP^n, \dbP^n) 
= J(\dbP',  \dbP),
\end{align*}
which, by the arbitrariness of $\dbP' \in \cR(\dbP)$, proves that $\dbP \in \cR^\star(\dbP)$ and thus ends the proof. 

\subsection{Proof of Theorem \ref{theorem:convergence}}

Without loss of generality, we may assume that $(\Omega^0, \cF_T^0, \dbF^0, \dbP^0)$ from \Cref{sec:N-player} is sufficiently large so that there exists  $\tilde \t$, $\tilde X$ such that $\dbP^0 \circ (\tilde \t, \tilde X)^{-1} = \hat m$, and a sequence $\{(\t_k, X^k)\}_{k \in\N^\smallfont{\star}}$ of $\P^0$-independent copies of $(\tilde \t, \tilde X)$, where the $\{X^k\}_{k \in\N^\smallfont{\star}}$ are driven by $\P^0$-independent $(\F,\P^0)$--Brownian motions $(W^k)_{k \in\N^\smallfont{\star}}$. For all $N \in\N^\star$, we denote $\btau^N \coloneqq  (\t_1, \dots, \t_N)$. 

\medskip
\no{\bf Step $1$.} Denote $\hat m^N \coloneqq  m^N( \btau^{N}, \bX^{N, \btau^{N}})$. Relying on the synchronous coupling method (initiated by \citeauthor*{mckean1967propagation} in \cite{mckean1967propagation}), we first prove that we may extract a subsequence (still denoted the same for notational simplicity) such that

\medskip
\no{$(i)$} $\cW_p\big( \hat m^N, \hat m \big) \underset{N \to \infty}{\longrightarrow} 0$, $\dbP^0$--a.s., 

\smallskip
\no{$(ii)$} $\dbE^{\dbP^\smallfont{0}}\Big[\lVert X_t^{k, N, \btau^N} - X_t^k \rVert \Big] \underset{N \to \infty}{\longrightarrow} 0$, for all $k \in [N]$ and $t\in[0,T]$.

\medskip
\no Since $b$ and $\si$ are Lipschitz-continuous (with respect to $\cW_p$ for the measure-valued variable) and the $X^k$ and $X^{k,N,\btau}$ have the same initial conditions, we easily see by using Hölder's, Burkholder--Davis--Gundy's and Gr\"onwall's inequalities that we have, for some constant $C \ge 0$
\begin{align}\label{ineq1}
\dbE^{\dbP^\smallfont{0}}\big[ \lVert X_{\cdot \wedge t}^{N, \btau^\smallfont{N}} - X_{\cdot \wedge t}^k \rVert_{\infty} \big] \le \int_0^t \dbE^{\dbP^\smallfont{0}}\big[\cW_p^p(\hat m_{\cdot \wedge s}^N, \hat m_{\cdot \wedge s})\big]\mathrm{d}s, \; \text{for all $N \in\N^\star$, $k \in [N]$ and $t \in [0,T]$},
\end{align}
where we denoted
\begin{align*}
\hat m_{\cdot \wedge s}^N \coloneqq  \frac 1 N \sum_{k=1}^N \d_{(\t_\smallfont{k}, X_{\smallfont{\cdot} \smallfont{\wedge} \smallfont{t}}^{\smallfont{k}\smallfont{,}\smallfont{N}\smallfont{,}\smallfont{\btau}})}, \; \hat m_{\cdot \wedge t} \coloneqq  \dbP^0 \circ (\tilde \t, \tilde X_{\cdot \wedge t})^{-1}.
\end{align*}
Introduce $\bar m_{\cdot \wedge t}^N \coloneqq  \frac 1 N \sum_{k=1}^N \d_{(\t_\smallfont{k}, X_{\smallfont{\cdot} \smallfont{\wedge} \smallfont{t}}^\smallfont{k})}$, $\bar m^N \coloneqq  \bar m^N_{\cdot \wedge T}$. We have
\begin{align*}
\cW_p^p(\hat m_{\cdot \wedge t}^N, \hat m_{\cdot \wedge t}) &\le p\big(\cW_p^p(\hat m_{\cdot \wedge s}^N, \bar m_{\cdot \wedge s}^N) + \cW_p^p(\bar m_{\cdot \wedge s}^N, \hat m_{\cdot \wedge s})\big)  \le p\bigg(\frac 1 N \sum_{k=1}^N \lVert X_{\cdot \wedge t}^{N, \btau^N} - X_{\cdot \wedge t}^k \rVert_{\infty} + \cW_p^p(\bar m_{\cdot \wedge s}^N, \hat m_{\cdot \wedge s})\bigg),
\end{align*}
where the second inequality comes from the fact that $\frac 1 N \sum_{k=1}^N \d_{ \{(\t_\smallfont{k}, X^{\smallfont{k}\smallfont{,}\smallfont{N}\smallfont{,} \smallfont{\btau}}), (\t_\smallfont{k}, X^\smallfont{k}) \}}$ is a coupling of $\hat m^N$ and $\hat m$. Therefore, by \Cref{ineq1} and Gr\"onwall's inequality, we have, for $t=T$
\begin{align*}
\dbE^{\dbP^\smallfont{0}}\big[\cW_p^p(\hat m^N, \hat m)\big] \le C'\dbE^{\dbP^\smallfont{0}}\big[\cW_p^p(\bar m^N, \hat m)\big],
\end{align*}
which implies that $\dbE^{\dbP^\smallfont{0}}[\cW_p^p(\hat m^N, \hat m)]\longrightarrow 0$, since $\cW_p^p(\bar m^N, \hat m) \longrightarrow 0$, $\dbP^0$--a.s., as $N$ goes to $\infty$, by the law of large numbers. This implies that $\cW_p^p(\hat m^N, \hat m)$ converges to $\hat m$ in $\P^0$-probability, and thus we may extract a subsequence such that this convergence holds $\dbP^0$--a.s., which proves $(i)$. Then $(ii)$ is a simple consequence of \Cref{ineq1}.

\medskip
\no{\bf Step $2$.} Denote $\cT_N$ the set of $[0,T]$-valued $\dbF^{W^\smallfont{1}, \dots, W^\smallfont{N}}$--stopping times, and $\cT_\infty \coloneqq  \cup_{N \in \N^\star} \cT_N$. Let us prove that 
\begin{align*}
\e_N &\coloneqq   \bigg( \sup_{\th \in \cT_{\smallfont{N}}} \dbE^{\dbP^\smallfont{0}}\Big[g\big(\th, X_\th^{1, N, \th \otimes_\smallfont{1} \btau^{\smallfont{N}\smallfont{,}\smallfont{-}\smallfont{1}}}, m^N( \th \otimes_1 \btau^{N,-1}, \bX^{N,  \th \otimes_\smallfont{1} \btau^{\smallfont{N}\smallfont{,}\smallfont{-}\smallfont{1}}})\big)\Big]- \dbE^{\dbP^\smallfont{0}}\Big[g\big(\t_1, X_{\t_\smallfont{1}}^{1, N, \btau}, m^N(\btau^{N}, \bX^{N, \btau})\big)\Big] \bigg)^+,
\end{align*}
goes to $0$ as $N \longrightarrow \infty$. By symmetry of the $(X^{1,N,\btau}, \dots, X^{N,N,\btau})$, we have
\[
V_0^{k,N} \le \dbE^{\dbP^\smallfont{0}}\Big[g\big(\t_1, X_{\t_\smallfont{1}}^{1, N, \btau}, m^N(\btau^{N}, \bX^{N, \btau})\big)\Big] + \e_N, \; \mbox{for all $N \in\N^\star$, $k \in [N]$.}
\]
 We first observe that $\e_N \le \b_N + \gamma_N$, where
\begin{align*}
\b_N &\coloneqq  \dbE^{\dbP^\smallfont{0}}\bigg[ \sup_{0 \le t \le T} \Big| g\big(t, X_t^{1, N, t \otimes_\smallfont{1} \btau^{\smallfont{N}\smallfont{,}\smallfont{-}\smallfont{1}}}, m^N( t \otimes_1 \btau^{N,-1}, \bX^{N,  t \otimes_\smallfont{1} \btau^{\smallfont{N}\smallfont{,}\smallfont{-}\smallfont{1}}})\big) - g(t, X_t^1, \hat \dbP ) \Big|\bigg], \\
\gamma_N &\coloneqq  \bigg( \sup_{\th \in \cT_\smallfont{\infty}} \dbE^{\dbP^\smallfont{0}}\big[g(\th, X_\th^1, \hat \dbP)\big] - \dbE^{\dbP^\smallfont{0}}\Big[g\big(\t_1, X_{\t_\smallfont{1}}^{1, N, \btau}, m^N(\btau^{N}, \bX^{N, \btau})\big)\Big] \bigg)^+.
\end{align*}
By the Lipschitz-continuity of $g$, we have for some constant $C \ge 0$
\begin{align*}
\b_N &\le C\dbE^{\dbP^\smallfont{0}}\bigg[ \sup_{0 \le t \le T} \Big\{ \big\| X_t^{1, t \otimes_\smallfont{1} \btau^{\smallfont{N}\smallfont{,}\smallfont{-}\smallfont{1}}} - X_t^1 \big\| + \cW_p\big( m^N\big( t \otimes_1 \btau^{N,-1}, \bX^{N,  t \otimes_\smallfont{1} \btau^{\smallfont{N}\smallfont{,}\smallfont{-}\smallfont{1}}}\big), \hat \dbP \big) \Big\} \bigg],
\end{align*}
which tends to 0 by \textbf{step $1$}. For the same reasons, we have
\begin{align*}
 \dbE^{\dbP^\smallfont{0}}\Big[g\big(\t_1, X_{\t_\smallfont{1}}^{1, N, \btau}, m^N(\btau^{N}, \bX^{N, \btau})\big)\Big] &\underset{N \to \infty}{\longrightarrow}  \dbE^{\dbP^\smallfont{0}}\big[g(\t_1, X_{\t_1}^{1},\hat \dbP )\big]  = \dbE^{\hat \dbP}\big[g(\t_1, X_{\t_1}^{1},\hat \dbP )\big],
 \end{align*}
and therefore $\gamma_N \longrightarrow 0$ by the optimality of $\hat \dbP$. Finally $\e_N \longrightarrow 0$. 

\medskip
\no{\bf Step $3$.} Our definition \ref{def:N-Nash} (ii) of $\e$--Nash equilibria requires stopping times for the Brownian filtration. Since the $\{\btau^N\}_{N \in\N^\smallfont{\star}}$ may not stopping times for $\dbF^{W^\smallfont{1}, \dots, W^\smallfont{N}}$, we need to approximate them by $\dbF^{W^\smallfont{1}, \dots, W^\smallfont{N}}$--stopping times. By \Cref{prop:pure-stopping}, for all $N \in\N^\star$, we may construct $(\btau^{N,n} \coloneqq  (\t_1^{n}, \dots, \t_N^{n}))_{n \in \N^\smallfont{\star}}$ in $(\cT_N)^N$ such that $\dbP^0 \circ (\btau^{N,n}, \bX^{N, \btau^{\smallfont{N}\smallfont{,}\smallfont{n}}})^{-1}$ converges to $\dbP^0 \circ (\btau^N, \bX^{N, \btau^\smallfont{N}})^{-1}$ for $\cW_p$ as $n \longrightarrow \infty$. Introduce
\begin{align*}
\e_N^n &\coloneqq   \sup_{(k,\th)\in[N]\times \cT_\smallfont{N}} \dbE^{\dbP^\smallfont{0}}\Big[g\big(\th, X_\th^{k, N, \th \otimes_\smallfont{k} \btau^{\smallfont{N}\smallfont{,}\smallfont{n}\smallfont{,}\smallfont{-}\smallfont{k}}}, m^N( \th \otimes_k \btau^{N,n,-k}, \bX^{N,  \th \otimes_\smallfont{k} \btau^{\smallfont{N}\smallfont{,}\smallfont{n}\smallfont{,}\smallfont{-}\smallfont{k}}})\big)-g\big(\t_k^{n}, X_{\t_\smallfont{k}^{\smallfont{n}}}^{k, N, \btau^{\smallfont{N}\smallfont{,}\smallfont{n}}}, m^N(\btau^{N,n}, \bX^{N, \btau^{\smallfont{N}\smallfont{,}\smallfont{n}}})\big)\Big] ^+.
\end{align*}
First, by continuity of $g$, we have
\begin{align*}
\max_{k \in [N]} \Big| \dbE^{\dbP^\smallfont{0}}\Big[g\big(\t_k^{n}, X_{\t_\smallfont{k}^{\smallfont{n}}}^{k, N, \btau^{\smallfont{N}\smallfont{,}\smallfont{n}}}, m^N(\btau^{N,n}, \bX^{N, \btau^{\smallfont{N}\smallfont{,}\smallfont{n}}})\big)\Big] - \dbE^{\dbP^\smallfont{0}}\Big[g\big(\t_k, X_{\t_\smallfont{k}}^{k, N, \btau^{N}}, m^N(\btau^{N}, \bX^{N, \btau^{\smallfont{N}v}})\big)\Big]  \Big| \underset{n \to \infty}{\longrightarrow} 0.
\end{align*}
Note also that 
\begin{align*}
\dbE^{\dbP^\smallfont{0}}\bigg[\sup_{0 \le t \le T}\Big| g\big(t, X_t^{k, N, t \otimes_\smallfont{k} \btau^{\smallfont{N}\smallfont{,}\smallfont{n}\smallfont{,}\smallfont{-}\smallfont{k}}}, m^N( t \otimes_k \btau^{N,n,-k}, \bX^{N,  t \otimes_\smallfont{k} \btau^{\smallfont{N}\smallfont{,}\smallfont{n}\smallfont{,}\smallfont{-}\smallfont{k}}})\big)- g\big(t, X_t^{k, N, t \otimes_\smallfont{k} \btau^{\smallfont{N}\smallfont{,}\smallfont{-}\smallfont{k}}}, m^N( t \otimes_k \btau^{N,-k}, \bX^{N,  t \otimes_\smallfont{k} \btau^{\smallfont{N}\smallfont{,}\smallfont{-}\smallfont{k}}})\big) \Big| \bigg],
\end{align*}
converges to $0$ as $n$ goes to $\infty$, uniformly in $k\in[N]$, since $g$ is Lipschitz-continuous in $(x,m)$, uniformly in $t \in [0,T]$, and using again \textbf{Step $1$}. Thus $\e_N^n \underset{n \to \infty}{\longrightarrow} \e^N$. Furthermore, since $(\btau^{N,n}, \bX^{N, \btau^{\smallfont{N}\smallfont{,}\smallfont{n}}})$ weakly to $(\btau^N, \bX^{N, \btau^\smallfont{N}})$, $m^N\big( \btau^{N,n}, \bX^{N, \btau^{\smallfont{N}\smallfont{,}\smallfont{n}}} \big)$ also converges weakly to $m^N\big( \btau^{N}, \bX^{N, \btau^{\smallfont{N}\smallfont{,}}} \big)$ as $n \to \infty$. Therefore, by \textbf{Step 1}, we may find a subsequence of $\e_{N}^{n_N}$-Nash equilibria such that $m^N\big( \btau^{N,n_N}, \bX^{N, \btau^{\smallfont{N}\smallfont{,}\smallfont{n_N}}} \big)$ converges weakly to $\hat m$, which is a constant in the space of measure-valued random variable. This implies that the convergence also holds in probability, and eventually $\dbP^0$-a.s. after passing to an appropriate subsequence. \qed

\begin{appendix}

\section{Technical lemmata}

\subsection{Compactness}

\begin{proposition}\label{prop:compact}
For all $m \in \cP_p(\Omega)$, $\cR(m)$ is $\cW_p$-compact. 
\end{proposition}
\begin{proof}
Fix $m \in \cP_p(\Omega)$. By compactness of $[0,T]$ and the fact that $\{\dbP_X : \dbP \in \cR(m)\}$ is a singleton, it is clear that $\cR(m)$ is tight. Let $(\dbP^n)_{n\in\N}$ be a sequence valued in $\cR(m)$, which weakly converges to some $\dbP$. Cleary, $\dbP$ satisfies the martingale problem characterising $\cR(m)$. By \Cref{rem:independent}, we also see that  $\cF_t^I$ is conditionally independent of $\cF_T$ given $\cF_t$ under $\dbP$. Therefore, $\dbP \in \cR(m)$.

\medskip
Note also that, since $\dbP^n_{X_\smallfont{0}} = \l \in \cP_p(\dbR^d)$ for all $n \in\N$, and that $b$ and $\si$ are Lipschitz-continuous in $x$ uniformly in $t$, then the $(\dbP^n)_{n \in\N}$---and, in fact, all the elements of $\cR(m)$---are $p$-uniformly integrable. Hence the convergence of $(\dbP^n)_{n\in\N}$ to $\dbP$ also holds for $\cW_p$ (see \citeauthor*{carmona2018probabilisticI} \cite[Theorem 5.5]{carmona2018probabilisticI}), and this completes the proof.
\end{proof}

\subsection{Density of pure stopping times}

For $m \in \cP_p(\Omega)$, we denote by $\cR^{\rm pure}(m)$ the set of $\dbP \in \cR(m)$ for which there exists a continuous $\Phi : \cC^d \longrightarrow [0,T]$ such that $ \t = \Phi(X)$, $\dbP$--a.s.

\begin{proposition}\label{prop:pure-stopping}
Assume {$\si \si^\top$} is invertible. For all $m \in \cP_p(\Omega)$, $\cR(m)$ is the $\cW_p$-closure of $\cR^{\rm pure}(m)$.
\end{proposition}
\begin{proof}
This is a straightforward adaptation of \cite[Theorem 6.4]{carmona2017mean} to our context, replacing their process $S$ with $X$. In particular, since {$\si \si^\top$} is invertible, {the kernel of $\si$ is reduced to $0$, and therefore} the distribution $\dbP_X$ is non-atomic for all $\dbP \in \cR(m)$, which allows us to use \cite[Proposition 6.2]{carmona2017mean}. Following the same procedure, we may thus construct a sequence of $\dbF^X$--stopping times $(\t^n)_{n \in \N}$ such that $\dbP_{(\t^n, X)}$ converges weakly to $\dbP$ as $n \to \infty$. Since this sequence is $p$-uniformly integrable, this implies that this convergence also holds for $\cW_p$, and therefore $\cR(m)$ is included in the $\cW_p$-closure of $\cR^{\rm pure}(m)$. As $\cR^{\rm pure}(m) \subset \cR(m)$ and $\cR(m)$ is $\cW_p$-closed by \Cref{prop:compact}, we easily see that the converse inclusion holds.
\end{proof}


\subsection{Hemicontinuity}

\begin{lemma}\label{lem:hemicontinuity}
Assume $b$ and $\si$ are bounded. Then the set-valued map $\cP_p(\Omega) \ni m \longmapsto \cR(m)$

\medskip
{$(i)$} has nonempty and compact values$;$

\smallskip
{$(ii)$} has a relatively compact range $\cP(\cP_p(\Omega));$

\smallskip
{$(iii)$} is hemicontinuous. 
\end{lemma}
\begin{proof}
$ (i)$ For all $m \in \cP_p(\Omega)$, it is clear that $\cR(m)$ is nonempty. We also know from \Cref{prop:compact} that it has compact values. 

\medskip
$(ii)$ By boundedness of $b$ and $\si$, the set $\{ \dbP_X : \dbP \in \cup_{m \in \cP_p(\Omega)} \cR(m) \}$ is a subset of the set of continuous semimartingale measures with bounded characteristics, which is compact (see \citeauthor*{meyer1984tightness} \cite[Theorem 4]{meyer1984tightness}). We easily deduce from this fact that $\cP(\cP_p(\Omega))$ is relatively compact. 

\medskip
$(iii)$ By $(ii)$, $\cP$ has a relatively compact range. Therefore, it is sufficient to prove that its graph is closed to prove its upper-hemicontinuity. Given two sequences $(m_n)_{n\in\N}$ valued $\cP_p(\Omega)$, and $(\dbP^n)_{n\in\N}$, such that for any $n\in\N$, $\dbP^n$ is valued in $\cP(m_n)$, converging respectively to some $m$ and $\dbP$, both in $\cP_p(\Omega)$, we easily see that the martingale problems in \Cref{def:P} pass to the limit, and then, along with \Cref{rem:independent}, that $\dbP \in \cR(m)$.

\medskip
We now prove the lower-hemicontinuity of $\cP$. Let $(m_n)_{n\in\N}$ be valued in $\cP_p(\Omega)$, and converging to $m \in \cP_p(\Omega)$. Fix also $\dbP \in \cR(m)$. We want to construct a sequence $(\dbP^n)_{n\in\N}$ such that for any $n\in\N$, $\P^n$ is valued in $\cP(m_n)$, and which converges to $\dbP$. {By the same arguments as in the proof of Theorem \ref{theorem:existence}, we may enlarge the canonical space and take} $(\tilde \Omega, \tilde \cF_T, \tilde \dbF \coloneqq  \{\tilde \cF_t\}_{0 \le t \le T}, \tilde \dbP)$ a filtered probability space endowed with $\tilde \t$, $\tilde X$ such that $\tilde \dbP_{(\tilde \t, \tilde X)} = \dbP$ and 
\[
 \tilde X_t = \tilde X_0 + \int_0^t b(s, \tilde X, m)\mathrm{d}s + \int_0^t \si(s, \tilde X, m)\mathrm{d}\tilde W_s, \; \mbox{$\tilde \dbP$--a.s.,} 
 \]
where $\tilde W$ is an $(\tilde \dbF,\tilde \dbP)$--Brownian motion. For $n \in \N$, we introduce $\tilde X^n$, defined as the unique strong solution of the SDE
\[
\tilde X_t^n = \tilde X_0 + \int_0^t b(s, \tilde X^n, m_n)\mathrm{d}s + \int_0^t \si(s, \tilde X^n, m_n)\mathrm{d}\tilde W_s. 
\]
We then define $\dbP^n \coloneqq  \tilde \dbP_{(\tilde \t, \tilde X^\smallfont{n})}$. It is clear that $\dbP^n$ satisfies the martingale problem of \Cref{def:P}. Moreover, since $\si$ is invertible, the fact that $\cF^I$ is independent of $\cF_T^{\tilde X}$ conditionally on $\cF_t$ implies that it is also the case under each $\dbP^n$, as the $\tilde X^n$ above are driven by the same Brownian motion as $\tilde X$. Therefore, $\dbP^n \in \cP(m_n)$ for all $n \in\N $. Note also that
\[
 \cW_p^p(\dbP^n, \dbP) = \cW_p^p\big(\tilde \dbP_{(\tilde \t, \tilde X)}, \tilde \dbP_{(\tilde \t, \tilde X^\smallfont{n})}\big) \le C_p\dbE^{\tilde \dbP}\big[ \lVert \tilde X^n - \tilde X \rVert_\infty^p \big],
 \] 
for some constant $C_p \ge 0$, and that the right-hand side term tends to $0$ as $n \to \infty$, by standard SDEs' estimates, due to the fact that $\cW_p(m_n, m) \longrightarrow 0$ as $n \longrightarrow \infty$. This provides the desired result.
 \end{proof}

\end{appendix}

\section*{Declarations}

\paragraph*{Conflict of interest} The authors declare that they have no conflict of interest.

\paragraph*{Competing Interests} The authors have no competing interests to declare that are relevant to the content of this article.


{\footnotesize
\begin{thebibliography}{70}
\providecommand{\natexlab}[1]{#1}
\providecommand{\url}[1]{\texttt{#1}}
\expandafter\ifx\csname urlstyle\endcsname\relax
  \providecommand{\doi}[1]{doi: #1}\else
  \providecommand{\doi}{doi: \begingroup \urlstyle{rm}\Url}\fi

\bibitem[Aliprantis and Border(2006)]{aliprantis2006infinite}
C.D. Aliprantis and K.~Border.
\newblock \emph{Infinite dimensional analysis: a hitchhiker's guide}.
\newblock Springer-Verlag Berlin Heidelberg, 3rd edition, 2006.

\bibitem[Assaf and Samuel-Cahn(1998)]{assaf1998optimal}
D.~Assaf and E.~Samuel-Cahn.
\newblock Optimal multivariate stopping rules.
\newblock \emph{Journal of Applied Probability}, 35\penalty0 (3):\penalty0
  693--706, 1998.

\bibitem[Aumann(1964)]{aumann1964mixed}
R.J. Aumann.
\newblock Mixed and behavior strategies in infinite extensive games.
\newblock In M.~Dresher, L.S. Shapley, and A.W. Tucker, editors, \emph{Advances
  in game theory}, volume~52 of \emph{Annals of mathematics studies}, pages
  627--650. Princeton University Press, 1964.

\bibitem[Baxter and Chacon(1977)]{baxter1977compactness}
J.R. Baxter and R.V. Chacon.
\newblock Compactness of stopping times.
\newblock \emph{Zeitschrift f{\"u}r Wahrscheinlichkeitstheorie und verwandte
  Gebiete}, 40\penalty0 (3):\penalty0 169--181, 1977.

\bibitem[Bensoussan and Friedman(1974)]{bensoussan1974nonlinear}
A.~Bensoussan and A.~Friedman.
\newblock Nonlinear variational inequalities and differential games with
  stopping times.
\newblock \emph{Journal of Functional Analysis}, 16\penalty0 (3):\penalty0
  305--352, 1974.

\bibitem[Bensoussan and Friedman(1977)]{bensoussan1977nonzero}
A.~Bensoussan and A.~Friedman.
\newblock Nonzero-sum stochastic differential games with stopping times and
  free boundary problems.
\newblock \emph{Transactions of the American Mathematical Society},
  231\penalty0 (2):\penalty0 275--327, 1977.

\bibitem[Bertucci(2018)]{bertucci2018optimal}
C.~Bertucci.
\newblock Optimal stopping in mean field games, an obstacle problem approach.
\newblock \emph{Journal de Math{\'e}matiques Pures et Appliqu{\'e}es},
  120:\penalty0 165--194, 2018.

\bibitem[Bertucci(2021)]{bertucci2021monotone}
C.~Bertucci.
\newblock Monotone solutions for mean field games master equations: finite
  state space and optimal stopping.
\newblock \emph{Journal de l'{\'E}cole Polytechnique---Math{\'e}matiques},
  8:\penalty0 1099--1132, 2021.

\bibitem[Billingsley(1999)]{billingsley1999convergence}
P.~Billingsley.
\newblock \emph{Convergence of probability measures}.
\newblock Wiley series in probability and statistics. John Wiley \& Sons Inc.,
  New York, 2nd edition, 1999.

\bibitem[Bismut(1977)]{bismut1977probleme}
J.-M. Bismut.
\newblock Sur un probl{\`e}me de {D}ynkin.
\newblock \emph{Zeitschrift f{\"u}r Wahrscheinlichkeitstheorie und Verwandte
  Gebiete}, 39:\penalty0 31--53, 1977.

\bibitem[Bouveret et~al.(2020)Bouveret, Dumitrescu, and
  Tankov]{bouveret2020mean}
G.~Bouveret, R.~Dumitrescu, and P.~Tankov.
\newblock Mean-field games of optimal stopping: a relaxed solution approach.
\newblock \emph{SIAM Journal on Control and Optimization}, 58\penalty0
  (4):\penalty0 1795--1821, 2020.

\bibitem[Bouveret et~al.(2022)Bouveret, Dumitrescu, and
  Tankov]{bouveret2022technological}
G.~Bouveret, R.~Dumitrescu, and P.~Tankov.
\newblock Technological change in water use: a mean-field game approach to
  optimal investment timing.
\newblock \emph{Operations Research Perspectives}, 9:\penalty0 100225, 2022.

\bibitem[Briani and Cardaliaguet(2018)]{briani2018stable}
A.~Briani and P.~Cardaliaguet.
\newblock Stable solutions in potential mean field game systems.
\newblock \emph{Nonlinear Differential Equations and Applications NoDEA},
  25\penalty0 (1):\penalty0 1--26, 2018.

\bibitem[Buckdahn and Engelbert(1984)]{buckdahn1984randomized}
R.~Buckdahn and H.-J. Engelbert.
\newblock Randomized stopping times: {D}oob's optiomal sampling theorem and
  optimal stopping.
\newblock \emph{Mathematische Nachrichten}, 115\penalty0 (1):\penalty0
  237--247, 1984.

\bibitem[Cardaliaguet et~al.(2019)Cardaliaguet, Delarue, Lasry, and
  Lions]{cardaliaguet2019master}
P.~Cardaliaguet, F.~Delarue, J.-M. Lasry, and P.-L. Lions.
\newblock \emph{The master equation and the convergence problem in mean field
  games}, volume 201 of \emph{Annals of mathematics studies}.
\newblock Princeton University Press, 2019.

\bibitem[Carmona and Delarue(2014)]{carmona2014master}
R.~Carmona and F.~Delarue.
\newblock The master equation for large population equilibriums.
\newblock In D.~Crisan, B.~Hambly, and T.~Zariphopoulou, editors,
  \emph{Stochastic analysis and applications 2014: in honour of Terry Lyons},
  volume 100 of \emph{Springer proceedings in mathematics and statistics},
  pages 77--128. Springer, 2014.

\bibitem[Carmona and Delarue(2018)]{carmona2018probabilisticI}
R.~Carmona and F.~Delarue.
\newblock \emph{Probabilistic theory of mean field games with applications
  {I}}, volume~83 of \emph{Probability theory and stochastic modelling}.
\newblock Springer International Publishing, 2018.

\bibitem[Carmona et~al.(2017)Carmona, Delarue, and Lacker]{carmona2017mean}
R.~Carmona, F.~Delarue, and D.~Lacker.
\newblock Mean field games of timing and models for bank runs.
\newblock \emph{Applied Mathematics \& Optimization}, 76\penalty0 (217--260),
  2017.

\bibitem[Cecchin and Delarue(2022)]{cecchin2022weak}
A.~Cecchin and F.~Delarue.
\newblock Weak solutions to the master equation of potential mean field games.
\newblock \emph{ArXiv preprint arXiv:2204.04315}, 2022.

\bibitem[Chalasani and Jha(2001)]{chalasani2001randomized}
P.~Chalasani and S.~Jha.
\newblock Randomized stopping times and {A}merican option pricing with
  transaction costs.
\newblock \emph{Mathematical Finance}, 11\penalty0 (1):\penalty0 33--77, 2001.

\bibitem[Dalang(1984)]{dalang1984arret}
R.C. Dalang.
\newblock Sur l'arr{\^e}t optimal de processus {\`a} temps multidimensionnel
  continu.
\newblock \emph{S{\'e}minaire de probabilit{\'e}s de Strasbourg},
  XVIII:\penalty0 379--390, 1984.

\bibitem[Delarue et~al.(2020)Delarue, Lacker, and Ramanan]{delarue2020master}
F.~Delarue, D.~Lacker, and K.~Ramanan.
\newblock From the master equation to mean field game limit theory: large
  deviations and concentration of measure.
\newblock \emph{The Annals of Probability}, 48\penalty0 (1):\penalty0 211--263,
  2020.

\bibitem[Dumitrescu et~al.(2021)Dumitrescu, Leutscher, and
  Tankov]{dumitrescu2021control}
R.~Dumitrescu, M.~Leutscher, and P.~Tankov.
\newblock Control and optimal stopping mean field games: a linear programming
  approach.
\newblock \emph{Electronic Journal of Probability}, 26\penalty0 (157):\penalty0
  1--49, 2021.

\bibitem[Dumitrescu et~al.(2022)Dumitrescu, Leutscher, and
  Tankov]{dumitrescu2022energy}
R.~Dumitrescu, M.~Leutscher, and P.~Tankov.
\newblock Energy transition under scenario uncertainty: a mean-field game
  approach.
\newblock \emph{ArXiv preprint arXiv:2210.03554}, 2022.

\bibitem[Dynkin(1969)]{dynkin1969game}
E.B. Dynkin.
\newblock A game-theoretic version of an optimal stopping problem.
\newblock \emph{Doklady Akademii Nauk SSSR}, 185\penalty0 (1):\penalty0 16--19,
  1969.

\bibitem[Dynkin and Yushkevich(1969)]{dynkin1969theorems}
E.B. Dynkin and A.A. Yushkevich.
\newblock \emph{Markov processes, theorems and problems}.
\newblock Plenum Press, 1969.

\bibitem[Ekren et~al.(2014)Ekren, Keller, Touzi, and Zhang]{ekren2014viscosity}
I.~Ekren, C.~Keller, N.~Touzi, and J.~Zhang.
\newblock On viscosity solutions of path dependent {PDE}s.
\newblock \emph{The Annals of Probability}, 42\penalty0 (1):\penalty0 204--236,
  2014.

\bibitem[Ekstr{\"o}m and Villeneuve(2006)]{ekstrom2006value}
E.~Ekstr{\"o}m and S.~Villeneuve.
\newblock On the value of optimal stopping games.
\newblock \emph{The Annals of Applied Probability}, 16\penalty0 (3):\penalty0
  1576--1596, 2006.

\bibitem[El~Karoui(1981)]{el1981aspects}
N.~El~Karoui.
\newblock Les aspects probabilistes du contr{\^o}le stochastique.
\newblock In \emph{\'Ecole d'{\'e}t{\'e} de probabilit{\'e}s de Saint-Flour
  IX--1979}, volume 876 of \emph{Lecture notes in mathematics}, pages 73--238.
  Springer, 1981.

\bibitem[El~Karoui et~al.(1997)El~Karoui, Peng, and Quenez]{el1997backward}
N.~El~Karoui, S.~Peng, and M.-C. Quenez.
\newblock Backward stochastic differential equations in finance.
\newblock \emph{Mathematical Finance}, 7\penalty0 (1):\penalty0 1--71, 1997.

\bibitem[Fan(1952)]{fan1952fixed}
K.~Fan.
\newblock Fixed-point and minimax theorems in locally convex topological linear
  spaces.
\newblock \emph{Proceedings of the National Academy of Sciences}, 38\penalty0
  (2):\penalty0 121--126, 1952.

\bibitem[Folland(1999)]{folland1999real}
G.B. Folland.
\newblock \emph{Real analysis: modern techniques and their applications},
  volume~40 of \emph{Pure and applied mathematics}.
\newblock John Wiley \& Sons, 2nd edition, 1999.

\bibitem[Friedman(1973)]{friedman1973stochastic}
A.~Friedman.
\newblock Stochastic games and variational inequalities.
\newblock \emph{Archive for Rational Mechanics and Analysis}, 51\penalty0
  (5):\penalty0 321--346, 1973.

\bibitem[Gangbo et~al.(2022)Gangbo, M{\'e}sz{\'a}ros, Mou, and
  Zhang]{gangbo2022mean}
W.~Gangbo, A.R. M{\'e}sz{\'a}ros, C.~Mou, and J.~Zhang.
\newblock Mean field games master equations with non-separable {H}amiltonians
  and displacement monotonicity.
\newblock \emph{The Annals of Probability}, 50\penalty0 (6):\penalty0
  2178--2217, 2022.

\bibitem[He et~al.(2023)He, Tan, and Zou]{he2023mean}
X.~He, X.~Tan, and J.~Zou.
\newblock A mean-field version of {B}ank--{E}l {K}aroui's representation of
  stochastic processes.
\newblock \emph{ArXiv preprint arXiv:2302.03300}, 2023.

\bibitem[Huang and Xie(2022)]{huang2022class}
J.~Huang and T.~Xie.
\newblock A class of mean-field games with optimal stopping and its inverse
  problem.
\newblock \emph{ArXiv preprint arXiv:2206.03095}, 2022.

\bibitem[Huang et~al.(2003)Huang, Caines, and Malham{\'e}]{huang2003individual}
M.~Huang, P.E. Caines, and R.P. Malham{\'e}.
\newblock Individual and mass behaviour in large population stochastic wireless
  power control problems: centralized and {N}ash equilibrium solutions.
\newblock In C.~Abdallah and F.~Lewis, editors, \emph{Proceedings of the 42nd
  IEEE conference on decision and control, 2003}, pages 98--103. IEEE, 2003.

\bibitem[Huang et~al.(2006)Huang, Malham{\'e}, and Caines]{huang2006large}
M.~Huang, R.P. Malham{\'e}, and P.E. Caines.
\newblock Large population stochastic dynamic games: closed-loop
  {M}c{K}ean--{V}lasov systems and the {N}ash certainty equivalence principle.
\newblock \emph{Communications in Information \& Systems}, 6\penalty0
  (3):\penalty0 221--252, 2006.

\bibitem[Huang et~al.(2007{\natexlab{a}})Huang, Caines, and
  Malham{\'e}]{huang2007invariance}
M.~Huang, P.E. Caines, and R.P. Malham{\'e}.
\newblock An invariance principle in large population stochastic dynamic games.
\newblock \emph{Journal of Systems Science and Complexity}, 20\penalty0
  (2):\penalty0 162--172, 2007{\natexlab{a}}.

\bibitem[Huang et~al.(2007{\natexlab{b}})Huang, Caines, and
  Malham{\'e}]{huang2007large}
M.~Huang, P.E. Caines, and R.P. Malham{\'e}.
\newblock Large-population cost-coupled {LQG} problems with nonuniform agents:
  individual-mass behavior and decentralized $\varepsilon$-{N}ash equilibria.
\newblock \emph{IEEE Transactions on Automatic Control}, 52\penalty0
  (9):\penalty0 1560--1571, 2007{\natexlab{b}}.

\bibitem[Huang et~al.(2007{\natexlab{c}})Huang, Caines, and
  Malham{\'e}]{huang2007nash}
M.~Huang, P.E. Caines, and R.P. Malham{\'e}.
\newblock The {N}ash certainty equivalence principle and {M}c{K}ean--{V}lasov
  systems: an invariance principle and entry adaptation.
\newblock In D.~Castanon and J.~Spall, editors, \emph{46th IEEE conference on
  decision and control, 2007}, pages 121--126. IEEE, 2007{\natexlab{c}}.

\bibitem[Kobylanski et~al.(2014)Kobylanski, Quenez, and Roger~de
  Campagnolle]{kobylanski2014dynkin}
M.~Kobylanski, M.-C. Quenez, and M.~Roger~de Campagnolle.
\newblock Dynkin games in a general framework.
\newblock \emph{Stochastics An International Journal of Probability and
  Stochastic Processes}, 86\penalty0 (2):\penalty0 304--329, 2014.

\bibitem[Kuhn(1953)]{kuhn1953extensive}
H.W. Kuhn.
\newblock Extensive games and the problem of information.
\newblock In H.W. Kuhn and A.W. Tucker, editors, \emph{Contributions to the
  theory of games, volume II}, volume~28 of \emph{Annals of mathematics
  studies}, pages 193--216. Princeton University Press, 1953.

\bibitem[Lasry and Lions(2006{\natexlab{a}})]{lasry2006jeux}
J.-M. Lasry and P.-L. Lions.
\newblock Jeux {\`a} champ moyen. {I}--{L}e cas stationnaire.
\newblock \emph{Comptes Rendus Math{\'e}matique}, 343\penalty0 (9):\penalty0
  619--625, 2006{\natexlab{a}}.

\bibitem[Lasry and Lions(2006{\natexlab{b}})]{lasry2006jeux2}
J.-M. Lasry and P.-L. Lions.
\newblock Jeux {\`a} champ moyen. {II}--{H}orizon fini et contr{\^o}le optimal.
\newblock \emph{Comptes Rendus Math{\'e}matique}, 343\penalty0 (10):\penalty0
  679--684, 2006{\natexlab{b}}.

\bibitem[Lasry and Lions(2007)]{lasry2007mean}
J.-M. Lasry and P.-L. Lions.
\newblock Mean field games.
\newblock \emph{Japanese Journal of Mathematics}, 2\penalty0 (1):\penalty0
  229--260, 2007.

\bibitem[Lepeltier and Maingueneau(1984)]{lepeltier1984jeu}
J.-P. Lepeltier and M.A. Maingueneau.
\newblock Le jeu de {D}ynkin en th{\'e}orie g{\'e}n{\'e}rale sans
  l'hypoth{\`e}se de {M}okobodski.
\newblock \emph{Stochastics: An International Journal of Probability and
  Stochastic Processes}, 13\penalty0 (1--2):\penalty0 25--44, 1984.

\bibitem[McKean(1967)]{mckean1967propagation}
Henry~P McKean.
\newblock Propagation of chaos for a class of non-linear parabolic equations.
\newblock \emph{Stochastic Differential Equations (Lecture Series in
  Differential Equations, Session 7, Catholic Univ., 1967)}, pages 41--57,
  1967.

\bibitem[Meyer(1978)]{meyer1978convergence}
P.-A. Meyer.
\newblock Convergence faible et compacit{\'e} des temps d'arr{\^e}t d'apr{\`e}s
  {B}axter et {C}hacon.
\newblock \emph{S\'eminaire de probabilit\'es de Strasbourg}, XII:\penalty0
  411--423, 1978.

\bibitem[Meyer and Zheng(1984)]{meyer1984tightness}
P.-A. Meyer and W.A. Zheng.
\newblock Tightness criteria for laws of semimartingales.
\newblock \emph{Annales de l'institut Henri Poincar{\'e}, Probabilit{\'e}s et
  Statistiques $({\mathrm{B}})$}, 20\penalty0 (4):\penalty0 353--372, 1984.

\bibitem[Milgrom and Roberts(58)]{milgrom1990rationalizability}
P.~Milgrom and J.~Roberts.
\newblock Rationalizability, learning, and equilibrium in games with strategic
  complementarities.
\newblock \emph{Econometrica}, 6:\penalty0 1255--1277, 58.

\bibitem[Mou and Zhang(2020)]{mou2020wellposedness}
C.~Mou and J.~Zhang.
\newblock Wellposedness of second order master equations for mean field games
  with nonsmooth data.
\newblock \emph{Memoirs of the American Mathematical Society}, to appear, 2020.

\bibitem[Nualart(1992)]{nualart1992randomized}
D.~Nualart.
\newblock Randomized stopping points and optimal stopping on the plane.
\newblock \emph{The Annals of Probability}, 20\penalty0 (2):\penalty0 883--900,
  1992.

\bibitem[Nutz(2018)]{nutz2018mean}
M.~Nutz.
\newblock A mean field game of optimal stopping.
\newblock \emph{SIAM Journal on Control and Optimization}, 56\penalty0
  (2):\penalty0 1206--1221, 2018.

\bibitem[Nutz et~al.(2020)Nutz, San~Mart\'in, and Tan]{nutz2020convergence}
M.~Nutz, J.~San~Mart\'in, and X.~Tan.
\newblock Convergence to the mean field game limit: a case study.
\newblock \emph{The Annals of Applied Probability}, 30\penalty0 (1):\penalty0
  259--286, 2020.

\bibitem[Pennanen and Perkki{\"o}(2018)]{pennanen2018optimal}
T.~Pennanen and A.-P. Perkki{\"o}.
\newblock Optimal stopping without snell envelopes.
\newblock \emph{ArXiv preprint arXiv:1812.04112}, 2018.

\bibitem[Peskir and Shiryaev(2006)]{peskir2006optimal}
G.~Peskir and A.N. Shiryaev.
\newblock \emph{Optimal stopping and free-boundary problems}.
\newblock Lectures in mathematics. ETH Z{\"u}rich. Birkh{\"a}user Basel, 2006.

\bibitem[Rosenberg et~al.(2001)Rosenberg, Solan, and
  Vieille]{rosenberg2001stopping}
D.~Rosenberg, E.~Solan, and N.~Vieille.
\newblock Stopping games with randomized strategies.
\newblock \emph{Probability Theory and Related Fields}, 119:\penalty0 433--451,
  2001.

\bibitem[Shiryaev(1978)]{shiryaev1978optimal}
A.N. Shiryaev.
\newblock \emph{Optimal stopping rules}, volume~8 of \emph{Stochastic modelling
  and applied probability}.
\newblock Springer-Verlag New York, 1978.

\bibitem[Shmaya and Solan(2014)]{shmaya2014equivalence}
E.~Shmaya and E.~Solan.
\newblock Equivalence between random stopping times in continuous time.
\newblock Technical report, Kellogg School of Management and Tel Aviv
  University, 2014.

\bibitem[Solan et~al.(2012)Solan, Tsirelson, and Vieille]{solan2012random}
E.~Solan, B.~Tsirelson, B., and N.~Vieille.
\newblock Random stopping times in stopping problems and stopping games.
\newblock Technical report, Tel Aviv University and HEC Paris, 2012.

\bibitem[Soner et~al.(2011)Soner, Touzi, and Zhang]{soner2011quasi}
H.M. Soner, N.~Touzi, and J.~Zhang.
\newblock Quasi-sure stochastic analysis through aggregation.
\newblock \emph{Electronic Journal of Probability}, 16\penalty0 (2):\penalty0
  1844--1879, 2011.

\bibitem[Talbi et~al.(2022)Talbi, Touzi, and Zhang]{talbi2022finite}
M.~Talbi, N.~Touzi, and J.~Zhang.
\newblock From finite population optimal stopping to mean field optimal
  stopping.
\newblock \emph{ArXiv preprint arXiv:2210.16004}, 2022.

\bibitem[Talbi et~al.(2023{\natexlab{a}})Talbi, Touzi, and
  Zhang]{talbi2023dynamic}
M.~Talbi, N.~Touzi, and J.~Zhang.
\newblock Dynamic programming equation for the mean field optimal stopping
  problem.
\newblock \emph{SIAM Journal on Control and Optimization}, to appear,
  2023{\natexlab{a}}.

\bibitem[Talbi et~al.(2023{\natexlab{b}})Talbi, Touzi, and
  Zhang]{talbi2023viscosity}
M.~Talbi, N.~Touzi, and J.~Zhang.
\newblock Viscosity solutions for obstacle problems on {W}asserstein space.
\newblock \emph{SIAM Journal on Control and Optimization}, 61\penalty0
  (3):\penalty0 1712--1736, 2023{\natexlab{b}}.

\bibitem[Tarski(1955)]{tarski1955lattice}
A.~Tarski.
\newblock A lattice-theoretical fixpoint theorem and its applications.
\newblock \emph{Pacific Journal of Mathematics}, 5\penalty0 (2):\penalty0
  285--309, 1955.

\bibitem[Touzi and Vieille(2002)]{touzi2002continuous}
N.~Touzi and N.~Vieille.
\newblock Continuous-time {D}ynkin games with mixed strategies.
\newblock \emph{SIAM Journal on Control and Optimization}, 41\penalty0
  (4):\penalty0 1073--1088, 2002.

\bibitem[Wu and Zhang(2020)]{wu2020viscosity}
C.~Wu and J.~Zhang.
\newblock Viscosity solutions to parabolic master equations and
  {M}c{K}ean--{V}lasov {SDE}s with closed-loop controls.
\newblock \emph{The Annals of Applied Probability}, 30\penalty0 (2):\penalty0
  936--986, 2020.

\bibitem[Xu and Zhou(2013)]{xu2013optimal}
Z.~Xu and X.Y. Zhou.
\newblock Optimal stopping under probability distortion.
\newblock \emph{The Annals of Applied Probability}, 23\penalty0 (1):\penalty0
  251--282, 2013.

\bibitem[Yasuda(1985)]{yasuda1985randomized}
M.~Yasuda.
\newblock On a randomized strategy in {N}eveu's stopping problem.
\newblock \emph{Stochastic Processes and their Applications}, 21\penalty0
  (1):\penalty0 159--166, 1985.

\end{thebibliography}


}
\end{document}